\documentclass[11pt]{article}
\usepackage[a4paper, margin=30mm]{geometry}

\usepackage{tikz}
\usepackage{amsthm}
\usepackage{thmtools}
\usepackage{amsmath}
\usepackage{amsfonts}
\usepackage[backend=biber,style=numeric,maxbibnames=99]{biblatex}
\usepackage{url}
\usepackage{hyperref}
\usepackage[hyphenbreaks]{breakurl}
\usepackage{epsfig}
\usepackage{float}
\usepackage{stmaryrd}
\usepackage[dvipsnames,svgnames,x11names]{xcolor}
\usepackage{graphicx}
\graphicspath{ {./drawings/} }

\usepackage{subfiles}

\DeclareFieldFormat[online]{title}{\mkbibquote{#1}}

\newtheorem{theorem}{Theorem}[section]
\declaretheorem[sibling=theorem]{conjecture}
\newtheorem{observation}[theorem]{Observation}
\newtheorem{case}{Case}
\newtheorem{lemma}[theorem]{Lemma}
\theoremstyle{definition}
\newtheorem{definition}[theorem]{Definition}

\newcommand{\unlab}[1]{\left\llbracket #1\right\rrbracket}
\newcommand{\cl}[0]{\texttt{cl}}

\title{Balanced bipartite distance of $K_4$-free graphs}
\author{J\'ozsef Balogh\footnotemark[1]\thanks{Department of Mathematics, University of Illinois Urbana-Champaign, Urbana, IL, USA, and Extremal Combinatorics and Probability Group (ECOPRO), Institute for Basic Science (IBS-R029-C4), Daejeon, South Korea. Partially supported by NSF grants RTG DMS-1937241, FRG DMS-2152488, (UIUC Campus Research Board Award RB26026), Simons Collaboration grant [SFI-MPS-TSM-00013107, JB].}
\and Ignacy Buczek\thanks{Faculty of Mathematics and Computer Science, Jagiellonian University, {\L}ojasiewicza 6, 30-348 Krak\'{o}w, Poland. Research supported by the National Science Centre grant 2021/42/E/ST1/00193.} \and Andrzej Grzesik\footnotemark[2] \and Piotr Kuc\footnotemark[2]}
\date{\today}

\begin{document}

\maketitle
\begingroup
\renewcommand{\thefootnote}{\relax}
\footnotemark
\footnotetext{E-mails: {\tt  jobal@illinois.edu}, {\tt ignacy.buczek@doctoral.uj.edu.pl}, {\tt andrzej.grzesik@uj.edu.pl}, {\tt piotr.kuc@student.uj.edu.pl}.}
\endgroup

\tikzset{flag/.style={scale=0.30}}

\tikzset{default/.style={
    draw,
    minimum size=1.8mm,
    inner sep=0pt,
    font=\tiny,
  }
}

\tikzset{root/.style={
    rectangle,
  }
}

\tikzset{vertex/.style={
    circle,
    fill=black,
  }
}

\tikzset{red/.style={
    fill=BrickRed,
    text=white
  }
}

\tikzset{blue/.style={
    fill=PowderBlue,
    text=black
  }
}

\tikzset{white/.style={
    fill=white,
    text=black
  }
}

\tikzset{black/.style={
    fill=black,
    text=white
  }
}

\def\FlagInit{
  \pgfdeclarelayer{roots}
  \pgfdeclarelayer{vertices}
  \pgfdeclarelayer{edges}
  \pgfsetlayers{edges, vertices, roots}
}

\def\VertexCoordinates#1#2{
  \pgfmathsetmacro{\N}{#1 - 1}
  \foreach \i in {0, ..., \N} {
    \pgfmathsetmacro{\angle}{360/#1 * \i + #2}%
    \pgfmathsetmacro{\x}{cos(\angle)}%
    \pgfmathsetmacro{\y}{sin(\angle)}%
    \coordinate (x\i) at (\x,\y);%
  }
}

\def\RootVertex#1#2#3{
  \begin{pgfonlayer}{roots}
    \node[default, root, #3] at (x#1) {\the\numexpr#2+1\relax};
  \end{pgfonlayer}
}

\def\DrawVertex#1#2{
  \begin{pgfonlayer}{vertices}
    \node[default, vertex, #2] at (x#1) {};
  \end{pgfonlayer}
}

\def\RootVertices#1#2{
  \pgfmathsetmacro{\N}{#1 - 1}
  \foreach \i in {0, ..., \N} {
    \RootVertex{\i}{\i}{#2}
  }
}

\def\Connect#1#2#3{
  \begin{pgfonlayer}{edges}
    \draw[#3] (x#1) -- (x#2);
  \end{pgfonlayer}
}

\def\Cycle#1{
  \pgfmathsetmacro{\N}{#1 - 1}
  \foreach \i in {0, ..., \N} {
    \pgfmathtruncatemacro{\nextvertex}{mod(\i + 1, #1)}
    \Connect{\i}{\nextvertex}{black}
  }
}

\newcommand{\vc}[1]{\ensuremath{\vcenter{\hbox{#1}}}}

\newcommand{\RootedVertex}[1]{
  \vc{
    \begin{tikzpicture}[flag]
      \FlagInit
      \VertexCoordinates{1}{0}
      \RootVertex{0}{0}{#1}
    \end{tikzpicture}
  }
}

\newcommand{\VertexM}[2]{
  \vc{
    \begin{tikzpicture}[flag]
      \FlagInit
      \VertexCoordinates{2}{240}
      \RootVertex{0}{0}{#1}
      \DrawVertex{1}{#2}
      \Connect{0}{1}{black}
    \end{tikzpicture}
  }
}

\newcommand{\VertexN}[2]{
  \vc{
    \begin{tikzpicture}[flag]
      \FlagInit
      \VertexCoordinates{2}{240}
      \RootVertex{0}{0}{#1}
      \DrawVertex{1}{#2}
    \end{tikzpicture}
  }
}

\newcommand{\VertexMM}[3]{
  \vc{
    \begin{tikzpicture}[flag]
      \FlagInit
      \VertexCoordinates{3}{270}
      \RootVertex{0}{0}{#1}
      \DrawVertex{1}{#2}
      \DrawVertex{2}{#3}
      \Connect{0}{1}{black}
      \Connect{0}{2}{black}
      \Connect{1}{2}{black}
    \end{tikzpicture}
  }
}

\newcommand{\VertexNM}[3]{
  \vc{
    \begin{tikzpicture}[flag]
      \FlagInit
      \VertexCoordinates{3}{270}
      \RootVertex{0}{0}{#1}
      \DrawVertex{1}{#2}
      \DrawVertex{2}{#3}
      \Connect{0}{1}{black}
      \Connect{1}{2}{black}
    \end{tikzpicture}
  }
}

\newcommand{\VertexNN}[3]{
  \vc{
    \begin{tikzpicture}[flag]
      \FlagInit
      \VertexCoordinates{3}{270}
      \RootVertex{0}{0}{#1}
      \DrawVertex{1}{#2}
      \DrawVertex{2}{#3}
      \Connect{1}{2}{black}
    \end{tikzpicture}
  }
}

\newcommand{\Edge}[2]{
  \vc{
    \begin{tikzpicture}[flag]
      \FlagInit
      \VertexCoordinates{2}{60}
      \DrawVertex{0}{#1}
      \DrawVertex{1}{#2}
      \Connect{0}{1}{black}
    \end{tikzpicture}
  }
}

\newcommand{\RootedEdge}[2]{
  \vc{
    \begin{tikzpicture}[flag]
      \FlagInit
      \VertexCoordinates{2}{180}
      \RootVertex{0}{0}{#1}
      \RootVertex{1}{1}{#2}
      \Connect{0}{1}{black}
    \end{tikzpicture}
  }
}

\newcommand{\Triangle}[3]{
  \vc{
    \begin{tikzpicture}[flag]
      \FlagInit
      \VertexCoordinates{3}{210}
      \DrawVertex{0}{#1}
      \DrawVertex{1}{#2}
      \DrawVertex{2}{#3}
      \Connect{0}{1}{black}
      \Connect{1}{2}{black}
      \Connect{0}{2}{black}
    \end{tikzpicture}
  }
}

\newcommand{\TriangleLM}[2]{
  \vc{
    \begin{tikzpicture}[flag]
      \FlagInit
      \VertexCoordinates{6}{210}
      \RootVertex{0}{0}{#1}
      \RootVertex{2}{1}{#1}
      \RootVertex{4}{2}{#1}
      \DrawVertex{5}{#2}
      \Connect{0}{2}{black}
      \Connect{0}{4}{black}
      \Connect{2}{4}{black}
      \Connect{0}{5}{black}
      \Connect{4}{5}{black}
    \end{tikzpicture}
  }
}

\newcommand{\TriangleMR}[2]{
  \vc{
    \begin{tikzpicture}[flag]
      \FlagInit
      \VertexCoordinates{6}{210}
      \RootVertex{0}{0}{#1}
      \RootVertex{2}{1}{#1}
      \RootVertex{4}{2}{#1}
      \DrawVertex{3}{#2}
      \Connect{0}{2}{black}
      \Connect{0}{4}{black}
      \Connect{2}{4}{black}
      \Connect{2}{3}{black}
      \Connect{3}{4}{black}
    \end{tikzpicture}
  }
}

\newcommand{\TriangleLR}[2]{
  \vc{
    \begin{tikzpicture}[flag]
      \FlagInit
      \VertexCoordinates{6}{210}
      \RootVertex{0}{0}{#1}
      \RootVertex{2}{1}{#1}
      \RootVertex{4}{2}{#1}
      \DrawVertex{1}{#2}
      \Connect{0}{2}{black}
      \Connect{0}{4}{black}
      \Connect{2}{4}{black}
      \Connect{0}{1}{black}
      \Connect{1}{2}{black}
    \end{tikzpicture}
  }
}

\newcommand{\EdgeL}[3]{
  \vc{
    \begin{tikzpicture}[flag]
      \FlagInit
      \VertexCoordinates{3}{210}
      \RootVertex{0}{0}{#1}
      \RootVertex{1}{1}{#2}
      \DrawVertex{2}{#3}
      \Connect{0}{1}{black}
      \Connect{0}{2}{black}
    \end{tikzpicture}
  }
}

\newcommand{\EdgeR}[3]{
  \vc{
    \begin{tikzpicture}[flag]
      \FlagInit
      \VertexCoordinates{3}{210}
      \RootVertex{0}{0}{#1}
      \RootVertex{1}{1}{#2}
      \DrawVertex{2}{#3}
      \Connect{0}{1}{black}
      \Connect{1}{2}{black}
    \end{tikzpicture}
  }
}

\newcommand{\EdgeN}[3]{
  \vc{
    \begin{tikzpicture}[flag]
      \FlagInit
      \VertexCoordinates{3}{210}
      \RootVertex{0}{0}{#1}
      \RootVertex{1}{1}{#2}
      \DrawVertex{2}{#3}
      \Connect{0}{1}{black}
    \end{tikzpicture}
  }
}

\newcommand{\EdgeB}[3]{
  \vc{
    \begin{tikzpicture}[flag]
      \FlagInit
      \VertexCoordinates{3}{210}
      \RootVertex{0}{0}{#1}
      \RootVertex{1}{1}{#2}
      \DrawVertex{2}{#3}
      \Connect{0}{1}{black}
      \Connect{0}{2}{black}
      \Connect{1}{2}{black}
    \end{tikzpicture}
  }
}

\newcommand{\EdgeNaN}[4]{
  \vc{
    \begin{tikzpicture}[flag]
      \FlagInit
      \VertexCoordinates{4}{225}
      \RootVertex{0}{0}{#1}
      \RootVertex{1}{1}{#2}
      \DrawVertex{2}{#3}
      \DrawVertex{3}{#4}
      \Connect{0}{1}{black}
      \Connect{2}{3}{black}
    \end{tikzpicture}
  }
}

\newcommand{\EdgeLN}[4]{
  \vc{
    \begin{tikzpicture}[flag]
      \FlagInit
      \VertexCoordinates{4}{225}
      \RootVertex{0}{0}{#1}
      \RootVertex{1}{1}{#2}
      \DrawVertex{2}{#3}
      \DrawVertex{3}{#4}
      \Connect{0}{1}{black}
      \Connect{0}{3}{black}
      \Connect{2}{3}{black}
    \end{tikzpicture}
  }
}

\newcommand{\EdgeRN}[4]{
  \vc{
    \begin{tikzpicture}[flag]
      \FlagInit
      \VertexCoordinates{4}{225}
      \RootVertex{0}{0}{#1}
      \RootVertex{1}{1}{#2}
      \DrawVertex{2}{#3}
      \DrawVertex{3}{#4}
      \Connect{0}{1}{black}
      \Connect{1}{2}{black}
      \Connect{2}{3}{black}
    \end{tikzpicture}
  }
}

\newcommand{\EdgeBN}[4]{
  \vc{
    \begin{tikzpicture}[flag]
      \FlagInit
      \VertexCoordinates{4}{225}
      \RootVertex{0}{0}{#1}
      \RootVertex{1}{1}{#2}
      \DrawVertex{2}{#3}
      \DrawVertex{3}{#4}
      \Connect{0}{1}{black}
      \Connect{0}{3}{black}
      \Connect{1}{3}{black}
      \Connect{2}{3}{black}
    \end{tikzpicture}
  }
}

\newcommand{\EdgeLaR}[4]{
  \vc{
    \begin{tikzpicture}[flag]
      \FlagInit
      \VertexCoordinates{4}{225}
      \RootVertex{0}{0}{#1}
      \RootVertex{1}{1}{#2}
      \DrawVertex{2}{#3}
      \DrawVertex{3}{#4}
      \Connect{0}{1}{black}
      \Connect{0}{3}{black}
      \Connect{1}{2}{black}
      \Connect{2}{3}{black}
    \end{tikzpicture}
  }
}

\newcommand{\EdgeBL}[4]{
  \vc{
    \begin{tikzpicture}[flag]
      \FlagInit
      \VertexCoordinates{4}{225}
      \RootVertex{0}{0}{#1}
      \RootVertex{1}{1}{#2}
      \DrawVertex{2}{#3}
      \DrawVertex{3}{#4}
      \Connect{0}{1}{black}
      \Connect{0}{2}{black}
      \Connect{1}{2}{black}
      \Connect{0}{3}{black}
      \Connect{2}{3}{black}
    \end{tikzpicture}
  }
}

\newcommand{\EdgeBR}[4]{
  \vc{
    \begin{tikzpicture}[flag]
      \FlagInit
      \VertexCoordinates{4}{225}
      \RootVertex{0}{0}{#1}
      \RootVertex{1}{1}{#2}
      \DrawVertex{2}{#3}
      \DrawVertex{3}{#4}
      \Connect{0}{1}{black}
      \Connect{1}{2}{black}
      \Connect{0}{3}{black}
      \Connect{1}{3}{black}
      \Connect{2}{3}{black}
    \end{tikzpicture}
  }
}

\newcommand{\EdgeLaL}[4]{
  \vc{
    \begin{tikzpicture}[flag]
      \FlagInit
      \VertexCoordinates{4}{225}
      \RootVertex{0}{0}{#1}
      \RootVertex{1}{1}{#2}
      \DrawVertex{2}{#3}
      \DrawVertex{3}{#4}
      \Connect{0}{1}{black}
      \Connect{0}{2}{black}
      \Connect{0}{3}{black}
      \Connect{2}{3}{black}
    \end{tikzpicture}
  }
}

\newcommand{\EdgeRaR}[4]{
  \vc{
    \begin{tikzpicture}[flag]
      \FlagInit
      \VertexCoordinates{4}{225}
      \RootVertex{0}{0}{#1}
      \RootVertex{1}{1}{#2}
      \DrawVertex{2}{#3}
      \DrawVertex{3}{#4}
      \Connect{0}{1}{black}
      \Connect{1}{2}{black}
      \Connect{1}{3}{black}
      \Connect{2}{3}{black}
    \end{tikzpicture}
  }
}

\newcommand{\RootCherry}[1]{
  \RootVertex{0}{0}{#1}
  \RootVertex{1}{1}{#1}
  \RootVertex{2}{2}{#1}
}

\newcommand{\CherryN}[2]{
  \vc{
    \begin{tikzpicture}[flag]
      \FlagInit
      \VertexCoordinates{4}{180}
      \RootCherry{#1}
      \DrawVertex{3}{#2}
      \Connect{0}{1}{black}
      \Connect{1}{2}{black}
    \end{tikzpicture}
  }
}

\newcommand{\CherryM}[2]{
  \vc{
    \begin{tikzpicture}[flag]
      \FlagInit
      \VertexCoordinates{4}{180}
      \RootCherry{#1}
      \DrawVertex{3}{#2}
      \Connect{0}{1}{black}
      \Connect{1}{2}{black}
      \Connect{1}{3}{black}
    \end{tikzpicture}
  }
}

\newcommand{\CherryR}[2]{
  \vc{
    \begin{tikzpicture}[flag]
      \FlagInit
      \VertexCoordinates{4}{180}
      \RootCherry{#1}
      \DrawVertex{3}{#2}
      \Connect{0}{1}{black}
      \Connect{1}{2}{black}
      \Connect{2}{3}{black}
    \end{tikzpicture}
  }
}

\newcommand{\CherryL}[2]{
  \vc{
    \begin{tikzpicture}[flag]
      \FlagInit
      \VertexCoordinates{4}{180}
      \RootCherry{#1}
      \DrawVertex{3}{#2}
      \Connect{0}{1}{black}
      \Connect{1}{2}{black}
      \Connect{0}{3}{black}
    \end{tikzpicture}
  }
}

\newcommand{\CherryLR}[2]{
  \vc{
    \begin{tikzpicture}[flag]
      \FlagInit
      \VertexCoordinates{4}{180}
      \RootCherry{#1}
      \DrawVertex{3}{#2}
      \Connect{0}{1}{black}
      \Connect{1}{2}{black}
      \Connect{0}{3}{black}
      \Connect{2}{3}{black}
    \end{tikzpicture}
  }
}

\newcommand{\CherryMR}[2]{
  \vc{
    \begin{tikzpicture}[flag]
      \FlagInit
      \VertexCoordinates{4}{180}
      \RootCherry{#1}
      \DrawVertex{3}{#2}
      \Connect{0}{1}{black}
      \Connect{1}{2}{black}
      \Connect{1}{3}{black}
      \Connect{2}{3}{black}
    \end{tikzpicture}
  }
}

\newcommand{\CherryLM}[2]{
  \vc{
    \begin{tikzpicture}[flag]
      \FlagInit
      \VertexCoordinates{4}{180}
      \RootCherry{#1}
      \DrawVertex{3}{#2}
      \Connect{0}{1}{black}
      \Connect{1}{2}{black}
      \Connect{0}{3}{black}
      \Connect{1}{3}{black}
    \end{tikzpicture}
  }
}

\newcommand{\CherryA}[2]{
  \vc{
    \begin{tikzpicture}[flag]
      \FlagInit
      \VertexCoordinates{4}{180}
      \RootCherry{#1}
      \DrawVertex{3}{#2}
      \Connect{0}{1}{black}
      \Connect{1}{2}{black}
      \Connect{0}{3}{black}
      \Connect{1}{3}{black}
      \Connect{2}{3}{black}
    \end{tikzpicture}
  }
}

\newcommand{\CherryLaL}[2]{
  \vc{
    \begin{tikzpicture}[flag]
      \FlagInit
      \VertexCoordinates{5}{198}
      \RootCherry{#1}
      \DrawVertex{3}{#2}
      \DrawVertex{4}{#2}
      \Connect{0}{1}{black}
      \Connect{1}{2}{black}
      \Connect{0}{3}{black}
      \Connect{0}{4}{black}
      \Connect{3}{4}{black}
    \end{tikzpicture}
  }
}

\newcommand{\CherryRaR}[2]{
  \vc{
    \begin{tikzpicture}[flag]
      \FlagInit
      \VertexCoordinates{5}{198}
      \RootCherry{#1}
      \DrawVertex{3}{#2}
      \DrawVertex{4}{#2}
      \Connect{0}{1}{black}
      \Connect{1}{2}{black}
      \Connect{2}{3}{black}
      \Connect{2}{4}{black}
      \Connect{3}{4}{black}
    \end{tikzpicture}
  }
}

\newcommand{\CherryLaR}[2]{
  \vc{
    \begin{tikzpicture}[flag]
      \FlagInit
      \VertexCoordinates{5}{198}
      \RootCherry{#1}
      \DrawVertex{3}{#2}
      \DrawVertex{4}{#2}
      \Connect{0}{1}{black}
      \Connect{1}{2}{black}
      \Connect{2}{3}{black}
      \Connect{0}{4}{black}
      \Connect{3}{4}{black}
    \end{tikzpicture}
  }
}

\newcommand{\CherryLRaR}[2]{
  \vc{
    \begin{tikzpicture}[flag]
      \FlagInit
      \VertexCoordinates{5}{198}
      \RootCherry{#1}
      \DrawVertex{3}{#2}
      \DrawVertex{4}{#2}
      \Connect{0}{1}{black}
      \Connect{1}{2}{black}
      \Connect{2}{3}{black}
      \Connect{0}{4}{black}
      \Connect{2}{4}{black}
      \Connect{3}{4}{black}
    \end{tikzpicture}
  }
}

\newcommand{\CherryLaLR}[2]{
  \vc{
    \begin{tikzpicture}[flag]
      \FlagInit
      \VertexCoordinates{5}{198}
      \RootCherry{#1}
      \DrawVertex{3}{#2}
      \DrawVertex{4}{#2}
      \Connect{0}{1}{black}
      \Connect{1}{2}{black}
      \Connect{0}{3}{black}
      \Connect{2}{3}{black}
      \Connect{0}{4}{black}
      \Connect{3}{4}{black}
    \end{tikzpicture}
  }
}

\newcommand{\CherryLRaLR}[2]{
  \vc{
    \begin{tikzpicture}[flag]
      \FlagInit
      \VertexCoordinates{5}{198}
      \RootCherry{#1}
      \DrawVertex{3}{#2}
      \DrawVertex{4}{#2}
      \Connect{0}{1}{black}
      \Connect{1}{2}{black}
      \Connect{0}{3}{black}
      \Connect{2}{3}{black}
      \Connect{0}{4}{black}
      \Connect{2}{4}{black}
      \Connect{3}{4}{black}
    \end{tikzpicture}
  }
}

\newcommand{\CherryNN}[2]{
  \vc{
    \begin{tikzpicture}[flag]
      \FlagInit
      \VertexCoordinates{5}{198}
      \RootCherry{#1}
      \DrawVertex{3}{#2}
      \DrawVertex{4}{#2}
      \Connect{0}{1}{black}
      \Connect{1}{2}{black}
      \Connect{3}{4}{black}
    \end{tikzpicture}
  }
}

\newcommand{\CherryNA}[2]{
  \vc{
    \begin{tikzpicture}[flag]
      \FlagInit
      \VertexCoordinates{5}{198}
      \RootCherry{#1}
      \DrawVertex{3}{#2}
      \DrawVertex{4}{#2}
      \Connect{0}{1}{black}
      \Connect{1}{2}{black}
      \Connect{0}{3}{black}
      \Connect{1}{3}{black}
      \Connect{2}{3}{black}
      \Connect{3}{4}{black}
    \end{tikzpicture}
  }
}

\newcommand{\CherryNM}[2]{
  \vc{
    \begin{tikzpicture}[flag]
      \FlagInit
      \VertexCoordinates{5}{198}
      \RootCherry{#1}
      \DrawVertex{3}{#2}
      \DrawVertex{4}{#2}
      \Connect{0}{1}{black}
      \Connect{1}{2}{black}
      \Connect{1}{3}{black}
      \Connect{3}{4}{black}
    \end{tikzpicture}
  }
}

\newcommand{\CherryMM}[2]{
  \vc{
    \begin{tikzpicture}[flag]
      \FlagInit
      \VertexCoordinates{5}{198}
      \RootCherry{#1}
      \DrawVertex{3}{#2}
      \DrawVertex{4}{#2}
      \Connect{0}{1}{black}
      \Connect{1}{2}{black}
      \Connect{1}{3}{black}
      \Connect{1}{4}{black}
      \Connect{3}{4}{black}
    \end{tikzpicture}
  }
}

\newcommand{\CherryMRaR}[2]{
  \vc{
    \begin{tikzpicture}[flag]
      \FlagInit
      \VertexCoordinates{5}{198}
      \RootCherry{#1}
      \DrawVertex{3}{#2}
      \DrawVertex{4}{#2}
      \Connect{0}{1}{black}
      \Connect{1}{2}{black}
      \Connect{2}{3}{black}
      \Connect{1}{4}{black}
      \Connect{2}{4}{black}
      \Connect{3}{4}{black}
    \end{tikzpicture}
  }
}

\newcommand{\CherryLMaR}[2]{
  \vc{
    \begin{tikzpicture}[flag]
      \FlagInit
      \VertexCoordinates{5}{198}
      \RootCherry{#1}
      \DrawVertex{3}{#2}
      \DrawVertex{4}{#2}
      \Connect{0}{1}{black}
      \Connect{1}{2}{black}
      \Connect{2}{3}{black}
      \Connect{0}{4}{black}
      \Connect{1}{4}{black}
      \Connect{3}{4}{black}
    \end{tikzpicture}
  }
}

\newcommand{\CherryAaR}[2]{
  \vc{
    \begin{tikzpicture}[flag]
      \FlagInit
      \VertexCoordinates{5}{198}
      \RootCherry{#1}
      \DrawVertex{3}{#2}
      \DrawVertex{4}{#2}
      \Connect{0}{1}{black}
      \Connect{1}{2}{black}
      \Connect{2}{3}{black}
      \Connect{0}{4}{black}
      \Connect{1}{4}{black}
      \Connect{2}{4}{black}
      \Connect{3}{4}{black}
    \end{tikzpicture}
  }
}

\newcommand{\CherryLaMR}[2]{
  \vc{
    \begin{tikzpicture}[flag]
      \FlagInit
      \VertexCoordinates{5}{198}
      \RootVertex{0}{1}{#1}
      \RootVertex{1}{0}{#1}
      \RootVertex{2}{2}{#1}
      \DrawVertex{3}{#2}
      \DrawVertex{4}{#2}
      \Connect{0}{1}{black}
      \Connect{1}{2}{black}
      \Connect{1}{3}{black}
      \Connect{2}{3}{black}
      \Connect{0}{4}{black}
      \Connect{3}{4}{black}
    \end{tikzpicture}
  }
}

\newcommand{\CherryLRaMR}[2]{
  \vc{
    \begin{tikzpicture}[flag]
      \FlagInit
      \VertexCoordinates{5}{198}
      \RootCherry{#1}
      \DrawVertex{3}{#2}
      \DrawVertex{4}{#2}
      \Connect{0}{1}{black}
      \Connect{1}{2}{black}
      \Connect{1}{3}{black}
      \Connect{2}{3}{black}
      \Connect{0}{4}{black}
      \Connect{2}{4}{black}
      \Connect{3}{4}{black}
    \end{tikzpicture}
  }
}

\newcommand{\CherryLaLM}[2]{
  \vc{
    \begin{tikzpicture}[flag]
      \FlagInit
      \VertexCoordinates{5}{198}
      \RootCherry{#1}
      \DrawVertex{3}{#2}
      \DrawVertex{4}{#2}
      \Connect{0}{1}{black}
      \Connect{1}{2}{black}
      \Connect{0}{3}{black}
      \Connect{1}{3}{black}
      \Connect{0}{4}{black}
      \Connect{3}{4}{black}
    \end{tikzpicture}
  }
}

\newcommand{\CherryLaA}[2]{
  \vc{
    \begin{tikzpicture}[flag]
      \FlagInit
      \VertexCoordinates{5}{198}
      \RootCherry{#1}
      \DrawVertex{3}{#2}
      \DrawVertex{4}{#2}
      \Connect{0}{1}{black}
      \Connect{1}{2}{black}
      \Connect{0}{3}{black}
      \Connect{1}{3}{black}
      \Connect{2}{3}{black}
      \Connect{0}{4}{black}
      \Connect{3}{4}{black}
    \end{tikzpicture}
  }
}

\newcommand{\CherryLMaLR}[2]{
  \vc{
    \begin{tikzpicture}[flag]
      \FlagInit
      \VertexCoordinates{5}{198}
      \RootCherry{#1}
      \DrawVertex{3}{#2}
      \DrawVertex{4}{#2}
      \Connect{0}{1}{black}
      \Connect{1}{2}{black}
      \Connect{0}{3}{black}
      \Connect{2}{3}{black}
      \Connect{0}{4}{black}
      \Connect{1}{4}{black}
      \Connect{3}{4}{black}
    \end{tikzpicture}
  }
}

\newcommand{\CherryLRaA}[2]{
  \vc{
    \begin{tikzpicture}[flag]
      \FlagInit
      \VertexCoordinates{5}{198}
      \RootCherry{#1}
      \DrawVertex{3}{#2}
      \DrawVertex{4}{#2}
      \Connect{0}{1}{black}
      \Connect{1}{2}{black}
      \Connect{0}{3}{black}
      \Connect{1}{3}{black}
      \Connect{2}{3}{black}
      \Connect{0}{4}{black}
      \Connect{2}{4}{black}
      \Connect{3}{4}{black}
    \end{tikzpicture}
  }
}

\newcommand{\CherryNaMR}[2]{
  \vc{
    \begin{tikzpicture}[flag]
      \FlagInit
      \VertexCoordinates{5}{198}
      \RootCherry{#1}
      \DrawVertex{3}{#2}
      \DrawVertex{4}{#2}
      \Connect{0}{1}{black}
      \Connect{1}{2}{black}
      \Connect{1}{3}{black}
      \Connect{2}{3}{black}
      \Connect{3}{4}{black}
    \end{tikzpicture}
  }
}

\newcommand{\CherryLMaN}[2]{
  \vc{
    \begin{tikzpicture}[flag]
      \FlagInit
      \VertexCoordinates{5}{198}
      \RootCherry{#1}
      \DrawVertex{3}{#2}
      \DrawVertex{4}{#2}
      \Connect{0}{1}{black}
      \Connect{1}{2}{black}
      \Connect{0}{4}{black}
      \Connect{1}{4}{black}
      \Connect{3}{4}{black}
    \end{tikzpicture}
  }
}

\newcommand{\CherryNaA}[2]{
  \vc{
    \begin{tikzpicture}[flag]
      \FlagInit
      \VertexCoordinates{5}{198}
      \RootCherry{#1}
      \DrawVertex{3}{#2}
      \DrawVertex{4}{#2}
      \Connect{0}{1}{black}
      \Connect{1}{2}{black}
      \Connect{0}{3}{black}
      \Connect{1}{3}{black}
      \Connect{2}{3}{black}
      \Connect{3}{4}{black}
    \end{tikzpicture}
  }
}

\newcommand{\CherryMaMR}[2]{
  \vc{
    \begin{tikzpicture}[flag]
      \FlagInit
      \VertexCoordinates{5}{198}
      \RootCherry{#1}
      \DrawVertex{3}{#2}
      \DrawVertex{4}{#2}
      \Connect{0}{1}{black}
      \Connect{1}{2}{black}
      \Connect{1}{3}{black}
      \Connect{2}{3}{black}
      \Connect{1}{4}{black}
      \Connect{3}{4}{black}
    \end{tikzpicture}
  }
}

\newcommand{\CherryLMaM}[2]{
  \vc{
    \begin{tikzpicture}[flag]
      \FlagInit
      \VertexCoordinates{5}{198}
      \RootCherry{#1}
      \DrawVertex{3}{#2}
      \DrawVertex{4}{#2}
      \Connect{0}{1}{black}
      \Connect{1}{2}{black}
      \Connect{1}{3}{black}
      \Connect{0}{4}{black}
      \Connect{1}{4}{black}
      \Connect{3}{4}{black}
    \end{tikzpicture}
  }
}

\newcommand{\CherryMaA}[2]{
  \vc{
    \begin{tikzpicture}[flag]
      \FlagInit
      \VertexCoordinates{5}{198}
      \RootCherry{#1}
      \DrawVertex{3}{#2}
      \DrawVertex{4}{#2}
      \Connect{0}{1}{black}
      \Connect{1}{2}{black}
      \Connect{0}{3}{black}
      \Connect{1}{3}{black}
      \Connect{2}{3}{black}
      \Connect{1}{4}{black}
      \Connect{3}{4}{black}
    \end{tikzpicture}
  }
}

\newcommand{\CherryLMaMR}[2]{
  \vc{
    \begin{tikzpicture}[flag]
      \FlagInit
      \VertexCoordinates{5}{198}
      \RootCherry{#1}
      \DrawVertex{3}{#2}
      \DrawVertex{4}{#2}
      \Connect{0}{1}{black}
      \Connect{1}{2}{black}
      \Connect{1}{3}{black}
      \Connect{2}{3}{black}
      \Connect{0}{4}{black}
      \Connect{1}{4}{black}
      \Connect{3}{4}{black}
    \end{tikzpicture}
  }
}

\begin{abstract}
  We show that every $K_4$-free graph on $n$ vertices can be made balanced bipartite by removing at most $\frac{n^2}{9}$ edges.
  This proves a conjecture of Balogh, Clemen, and Lidick\'y, and generalizes both Sudakov's result on the bipartite distance of $K_4$-free graphs and Reiher's result on the sparse half of $K_4$-free graphs.
\end{abstract}

\section{Introduction}
The famous \emph{Sparse half conjecture} of Erd\H{o}s~\cite{Original}, open for over 50 years, states that every $n$-vertex triangle-free graph has a subset of  vertices of size   $\frac{n}{2}$ that spans at most $\frac{n^2}{50}$ edges. 
This conjecture received a lot of attention, see  \cite{erdosK4half, tri-sparse-first, tri-dense, tri-sparse}, with Razborov achieving the strongest known general bound of $\frac{27n^2}{1024}$ in \cite{tri-sparse}.

A similar problem, the \emph{Sparse halves conjecture}, which  asks for the minimum number of edges that must be removed from a triangle-free graph to make it bipartite, was posed by Erd\H{o}s~\cite{erdosHalves} even earlier.
This question is closely related to the \emph{Sparse half conjecture}, as it is expected that the tight lower bounds for both problems are attained by the same extremal graph (a balanced blowup of a pentagon).
Because of this, it is believed that to make a triangle-free graph bipartite, it is always sufficient to remove $\frac{n^2}{50} \cdot 2 = \frac{n^2}{25}$ edges.
This problem is also difficult and has attracted considerable attention \cite{sparse-halves-best, erdosK4halves, make-bipartite-2}.
The best known general bound is $\frac{2n^2}{47}$, which   was recently proved in \cite{sparse-halves-best}.

It is important to note that there is no implication relationship between the \emph{Sparse half conjecture} and the \emph{Sparse halves conjecture}, which is why they gave rise to two relatively independent research directions.
A pertinent question is that if these problems have an intuitive generalization, with a natural candidate being a hypothetical bound on the number of edges that have to be removed to make a triangle-free graph balanced bipartite.
If the tight bound was identical to the \emph{Sparse halves conjecture} and equaled $\frac{n^2}{25}$, it would imply both of the conjectures.
This question was initially asked and answered by Balogh, Clemen, and Lidick\'y in \cite{10problems}, with the answer being unfortunately negative.
It is easy to see that to make the complete bipartite graph with parts of sizes $\frac{3n}{4}$ and $\frac{n}{4}$ balanced bipartite, one always has to remove at least $\frac{n^2}{16}$ edges.
Balogh, Clemen, and Lidick\'y (\cite{10problems}) proved that this is the optimal bound in general.
This implies that, surprisingly, proving a tight bound for this problem is significantly easier than for the unbalanced one.

\begin{theorem}[Balogh, Clemen, and Lidick\'y \cite{10problems}]\label{triangle-free-impossible}
  Every triangle-free graph on $n$ vertices can be made balanced bipartite by removing at most $\frac{n^2}{16}$ edges.
\end{theorem}

The motive of asking how far from being bipartite a graph can be is a recurring theme  in Erd\H{o}s' work.
Not long after asking this question for the class of triangle-free graphs, he~ \cite{erdosK4halves} formulated a similar problem in the $K_4$-free setting, where
he conjectured that the complete tripartite graph is the extremal example.
 This was proved by Sudakov~ \cite{bipartite}.

\begin{theorem}[Sudakov \cite{bipartite}]\label{sudakov}
  Every $K_4$-free graph on $n$ vertices can be made bipartite by removing at most $\frac{n^2}{9}$ edges.
\end{theorem}

The \emph{Sparse half conjecture} also has a natural variant for  $K_4$-free graphs, see \cite{erdosK4halves}, which was posed  by Chung and Graham~\cite{Chung} and independently by Erd\H os, Faudree, Rousseau and Schelp~\cite{erdosK4half}.
Both groups anticipated that the tight bound should be once again attained by the complete balanced tripartite graph.
Confirming this bound seemed more elusive than proving the tight bound on the bipartite distance, but building on the partial progress of Liu and Ma (\cite{liu-ma}), who settled the conjecture for regular graphs,  Reiher~\cite{sparse} recently managed to do so.

\begin{theorem}[Reiher \cite{sparse}]\label{reiher}
  Every $K_4$-free graph on $n$ vertices has a subset of  $\frac{n}{2}$ vertices which spans  at most $\frac{n^2}{18}$ edges.
\end{theorem}

Analogously to the \emph{Sparse half conjecture} and the \emph{Sparse halves conjecture},
 \autoref{sudakov} and \autoref{reiher} could be potentially generalized by a sufficiently strong bound on the balanced bipartite distance in the class of $K_4$-free graphs.
\autoref{triangle-free-impossible} naturally suggests that such a bound is unlikely to be obtained.
Intriguingly, Balogh, Clemen, and Lidick\'y~\cite{10problems} noticed that the usual interesting candidates in the $K_4$-free class are not sufficient to exclude the possibility of such a bound, which led them to conjecture that it might actually be achievable.

\begin{conjecture}[Balogh, Clemen, and Lidick\'y \cite{10problems}]\label{main}
  Every $K_4$-free graph on $n$ vertices can be made balanced bipartite by removing at most $\frac{n^2}{9}$ edges.
\end{conjecture}

In this work, we show that \autoref{main} holds.
This directly generalizes, and gives a different proof of the aforementioned results of Sudakov (\cite{bipartite}) and Reiher (\cite{sparse}).
Moreover, it confirms an unexpected discrepancy between the classes of triangle-free and $K_4$-free graphs.

The proof consists of  two major steps. 
First, we show \autoref{main} for all graphs that are sufficiently close to the complete balanced tripartite graph.
In the second step, we introduce an improved approach of describing graphs cuts in Razborov's flag algebra framework (\cite{flag-algebras}) to show that a potential counterexample would need to be very close to the complete balanced tripartite graph, which is not possible due to the previously proven statement. 

It is worth noting that while the idea of combining flag algebras with graph cuts has been used in the literature (\cite{10problems}, \cite{sparse-halves-best}), its usual application is not suitable for this problem or most partitioning problems that impose requirements on the sizes of the components in the partition.
We overcome this obstacle by applying flag algebras to a vertex-colored graph derived from the potential counterexample, where colors of vertices encode extra information that gives more control of the partitions.
We believe that this approach should be applicable to other partitioning problems and should lead to further progress on the \emph{Sparse half conjecture} and \emph{Sparse half conjectures}.

The paper is organized as follows.
In \autoref{sec:preliminaries} we introduce basic graph notation used in the paper.
\autoref{sec:simple} is devoted to proving some basic subcases of \autoref{main} that will be needed later.
In \autoref{sec:2-ind} we include the proof of \autoref{main} for graphs that contain a large pair of independent sets.
This result is used in \autoref{sec:almost} that solves the conjecture in the close neighborhood of the extremal example.
To finish the proof in the general case, we use the flag algebra framework, which we briefly introduce in \autoref{sec:flags}.
The proof of our main result  is completed in \autoref{sec:general}.

\section{Notation}\label{sec:preliminaries}

For a graph~$G$ and a vertex $v \in V(G)$, we define the \emph{open neighborhood} of $v$ in $G$ as
	$$N(v) := \{u \in V(G) : (u, v) \in E(G)\}.$$
We define the \emph{degree} of $v$ in $G$ as $\deg(v) := |N(v)|$ and the \emph{minimum degree} of $G$ as $\delta(G) := \min_{v \in V(G)} \deg(v).$

Similarly, for a pair of vertices $u, v \in V(G)$, we define the \emph{open common neighborhood} of $u$ and $v$ in $G$ as
	$$N(u, v) := \{w \in V(G) : (u, w), (v, w) \in E(G)\},$$
and the \emph{common degree} of $u, v$ in $G$ as $\deg(u, v) := |N(u, v)|$.	

For a graph~$G$ and a set of vertices $U \subseteq V(G)$, we define the \emph{open neighborhood} of $U$ in $G$ as
	$$N(U) := \bigcup \{N(v) : v \in U\} \setminus U.$$
We define $G[U]$ as the subgraph of $G$ induced on $U$, that is, a graph with vertex set $U$ and edge set $\{(u, v) \in E(G) : u, v \in U\}$.
If $U$ is equal to $V(G) \setminus \{ v \}$ for some $v \in V(G)$, we simply write $G - v$.
For simplicity, we denote $E(U) := E(G[U])$ and $e(U) := |E(U)|$. Whenever the graph $G$ is not clear from the context, we add subscript $G$ and write $e_G(U)$.

Analogously, for a graph~$G$ and sets of vertices $U, W \subseteq V(G)$, we denote the set of edges of $G$ that are between $U$ and $W$ as
	$$E(U, W) := \{(u, v) \in E(G) : u \in U, v \in W \}$$
and their number by 
	$$e(U, W) := |E(U, W)|.$$
If $U = \{v\}$ for some $v \in V(G)$, we use a short-hand syntax $E(v, W)$, and $e(v, W)$.


A \emph{$k$-vertex-colored graph} is a pair $(G, \theta)$, where $G$ is a graph and $\theta$ is a mapping from $V(G)$ to $[k]:=\{1, \ldots, k\}$.
If the mapping $\theta$ is clear from the context, we refer to the $k$-vertex-colored graph $(G, \theta)$ as $G$.


\begin{definition}
	For a graph~$G$ we define a \emph{$k$-blowup} of $G$ as the following graph~$H$:
	\begin{itemize}
		\item $V(H) := \bigcup \{ \{v_1, \ldots, v_k\}: v \in V(G)\}$,
		\item $E(H) := \bigcup \{ \{ (u_i, v_j): 1 \leq i, j \leq k \}: (u, v) \in E(G)\}$.
	\end{itemize}
  For an $\ell$-vertex-colored graph $(G, \theta)$, we define its $k$-blowup as an $\ell$-vertex-colored graph $(H, \theta')$, where $H$ is the $k$-blowup of $G$ and the mapping $\theta'$ is defined as
  $\theta'(v_i) = \theta(v)$ for every $1 \leq i \leq k$ and every $v \in V(G)$.
\end{definition}

\begin{definition}
	For a graph~$G$, we define the \emph{bipartite distance} of $G$ as
	$$\texttt{b}(G) := \min_{A \subseteq V(G)} e(A) + e(V(G) \setminus A).$$
\end{definition}

\begin{definition}
	For a graph~$G$, we define the \emph{balanced bipartite distance} of $G$ as
	$$\texttt{bb}(G) := \min_{\substack{A \subseteq V(G) \\ |A| = \lfloor \frac{n}{2} \rfloor}} e(A) + e(V(G) \setminus A).$$
\end{definition}

Intuitively, $\texttt{b}(G)$ is the number of edges one has to remove to make $G$ bipartite, while $\texttt{bb}(G)$ is the number of edges one has to remove to make $G$ balanced bipartite.

\begin{definition}
	For a graph~$G$ and a partition $(A, B)$ of $V(G)$, we say that an edge~${(u, v)}$ of $G$ is a \emph{class-edge} in the partition $(A, B)$, if $\{u, v\} \subseteq A$ or $\{u, v\} \subseteq B$.
\end{definition}

\section{Simple cases}\label{sec:simple}
\subsection{Divisibility by six}
When arguing about balanced bipartitions, divisibility of the number of vertices plays a key role.
In most of our arguments, we frequently divide the number of vertices by $2$ or~$3$.
For that reason, we devote the first couple of arguments to showing that we only need to consider graphs, whose number of vertices is a multiply of $6$.

\begin{lemma}\label{2-blowup}
	If \autoref{main} holds for a $2$-blowup of a graph~$G$, then it also holds for~$G$.
\end{lemma}

\begin{proof}
	Let $n$ be the order of $G$, and let $H$ be the 2-blowup of $G$.
	From the assumption, there exists a balanced bipartition $(A, B)$ of $H$ with at most $4n^2/9$
	 class-edges.
	W.l.o.g.~we may assume that for any $v \in V(G)$, whose copies belong to different parts of the partition $(A, B)$, it holds $v_1 \in A$ and $v_2 \in B$.

	Now, assume there exist two different vertices $u, v \in V(G)$, whose copies are in different parts of the partition.
	Let $(A', B')$ be the partition obtained by swapping $u_2$ with $v_1$ in $(A, B)$.
	Notice that $e(u_1, A') = e(u_2, A')$ and $e(u_1, A') \leq e(u_1, A)$, because $u_1$ and $u_2$ are surely disconnected and $u_1$ might have been connected to $v_1$.
	Using this observation, we can form an upper bound on the number of class-edges adjacent to any copy of $u$ or $v$ in the partition $(A', B')$
	$$e(u_1, A') + e(u_2, A') + e(v_1, B) + e(v_2, B) \leq 2e(u_1, A) + 2e(v_2, B).$$
	Similarly, we can consider an analogous partition obtained by swapping $u_1$ with $v_2$, and justify that all copies of $u$ and $v$ would have at most $2e(v_1, A) + 2e(u_2, B)$ adjacent class-edges.
	In the partition $(A, B)$, the number of class-edges adjacent to every copy of $u$ or $v$ is exactly
	$$e(u_1, A) + e(v_1, A) + e(u_2, B) + e(v_2, B),$$
	so at least one of the swaps does not increase the number of class-edges.
	Hence, we may  assume that all but at most one vertex of $G$ have both of their copies on one side of the partition $(A, B)$.

	Define $A_G$ as the set of vertices of $G$ that have both of their copies in $A$, and analogously define $B_G$.
	If $(A_G, B_G)$ is a partition of $V(G)$, then it is easy to see that it has at most
	$$e(A_G) + e(B_G) = \frac{1}{4}(e(A) + e(B)) = \frac{1}{4} \cdot \frac{4n^2}{9} = \frac{n^2}{9}$$
	class-edges.
	Now, assume there exists a vertex $v$ in $V(G) \setminus (A_G \cup B_G)$ (we justified there can be at most one such vertex).
	Notice that
	$$e(A_G) + e(B_G) + \frac{1}{2}e(v, A_G) + \frac{1}{2}e(v, B_G) = \frac{1}{4}(e(A) + e(B)).$$
	As a result, if we assign $v$ randomly to $A_G$ or $B_G$, we can obtain a balanced bipartition with at most $\frac{1}{4} \cdot \frac{4n^2}{9} = \frac{n^2}{9}$ class-edges.
\end{proof}

\begin{lemma}\label{3-blowup}
	If \autoref{main} holds for a $3$-blowup of a graph~$G$, then it also holds for~$G$.
\end{lemma}

\begin{proof}
	Let $n$ be the order of $G$ and let $H$ be a $3$-blowup of $G$.
	From the assumption, there exists a balanced bipartition $(A, B)$ of $H$ with at most $n^2$ class-edges.
	This time, $|A|$ and $|B|$ have to be divisible by $3$, which implies that either zero or at least two vertices of $G$ can have their copies on different sides.
	Let $u$ and $v$ be two vertices of $G$ that have their copies on two different sides of the partition $(A, B)$.
	\begin{case}
		Both vertices $u$ and $v$ have the same number of copies in $A$.
	\end{case}
	W.l.o.g., we can assume that they have exactly one copy in $A$, and those copies are $u_1$ and $v_1$.
	Consider swaps $(u_1, v_2)$ and $(u_2, v_1)$ as in the proof of \autoref{2-blowup}.
	At least one of these swaps does not increase the number of class-edges and both of these swaps increase the number of vertices that have all of their copies on one side of the partition.
	Hence, in this case, we can make a swap to decrease the number of vertices of $G$ with their copies on different sides of the partition $(A, B)$ without increasing the number of class-edges.

	\begin{case}
		Vertices $u$ and $v$ have a different number of copies in $A$.
	\end{case}
	Assume, w.l.o.g., that $u_1$ is the only copy of $u$ in $A$ and $v_2$, $v_3$ are the copies of $v$ in~$A$.
	Consider the natural swap of $u_1$ with $v_1$ and the less intuitive swap of $(u_2, u_3)$ with $(v_2, v_3)$.
	Analogously to the 2-blowup case, there are at most $3e(v_2, A) + 3e(u_2, B)$ class-edges adjacent to any copy of $u$ or $v$ in the partition created by the first swap,
	at most $3e(u_1, A) + 3e(v_1, B)$ in the partition created by the second swap, and exactly $2e(v_2, A) + e(u_1, A) + 2e(u_2, B) + e(v_1, B)$ in the partition $(A, B)$.
	It holds
	$$2(3e(v_2, A) + 3e(u_2, B)) + 3e(u_1, A) + 3e(v_1, B)$$
	$$ = 3(2e(v_2, A) + e(u_1, A) + 2e(u_2, B) + e(v_1, B)), $$
	which implies that one of these swaps does not increase the number of class-edges in $(A, B)$.
	Hence, we can assume that no vertices of $G$ have their copies in both $A$ and $B$.
	We can now consider a balanced bipartition $(A_G, B_G)$ of $G$, which is simply constructed by arranging vertices according to where all of their copies are in $(A, B)$.
	It is easy to see that $(A_G, B_G)$ contains at most $\frac{1}{9} \cdot n^2 = \frac{n^2}{9}$ class-edges.
\end{proof}

\begin{theorem}\label{parity}
	If \autoref{main} holds for every $K_4$-free graph on $6n$ vertices, then it also holds for every $K_4$-free graph on $n$ vertices.
\end{theorem}

\begin{proof}
	Let $G$ be a potential counterexample to \autoref{main} on $n$ vertices.
	Let $H$ be a $2$-blowup of~$G$, and let $F$ be a $3$-blowup of~$H$.
	The order of $F$ is  divisible by $6$, so it has a balanced partition with at most $4n^2$ class-edges.
	By \autoref{3-blowup}, $H$ has a balanced partition with at most ${4n^2}/{9}$ class-edges and by \autoref{2-blowup}, $G$ has a balanced partition with at most ${n^2}/{9}$ class-edges.
\end{proof}

\subsection{Vertices of a small degree}
We now show that  small degree vertices  are reducible.

\begin{observation}\label{min-degree-obs}
  If \autoref{main} holds for every $K_4$-free on $n$ vertices with minimum degree at least $\frac{4n - 1}{9}$, then it also holds for every $K_4$-free graph.
\end{observation}
\begin{proof}
  Assume, for contradiction, that \autoref{main} holds for every $n$-vertex $K_4$-free graph with  minimum degree at least $\frac{4n - 1}{9}$, but not for all $K_4$-free graphs.
  Let $G$ be the smallest (with respect to the size of the vertex set) $K_4$-free graph $G$ which  does not satisfy \autoref{main}. Then by \autoref{2-blowup} its
   $2$-blowup also does not satisfy \autoref{main}.
By our assumption,   $G$  has at least one vertex $v$ of degree at most ${(4n - 2)}/{9}$, where $n$ is the number of vertices in $G$.
  The minimality of $G$ implies that the $2$-blowup of $G - v$  has a balanced bipartition $(A, B)$ with at most 
  ${(2n - 2)^2}/{9}$ class-edges.
  By randomly assigning $v_1$ to $A$ or $B$, and $v_2$ to the other part, we maintain balancedness and obtain a randomized balanced bipartition of the 2-blowup of $G$, which in expectation has at most
  \[
    \frac{(2n - 2)^2}{9} + 2 \cdot \frac{1}{2} \cdot \frac{2(4n - 2)}{9} = \frac{4n^2}{9}
  \]
  class-edges.
  This implies that there exists a sufficiently sparse balanced bipartition of the 2-blowup of $G$, and hence of $G$, a contradiction.
\end{proof}

%

\subsection{Tripartite graphs}

\autoref{main} has been proven in \cite{10problems} to hold for tripartite graphs, and prior to this paper, this was the strongest known partial result on the conjecture.
We provide a slightly simplified proof, whose main idea is also relevant in subsequent cases.

\begin{theorem}[Balogh, Clemen and Lidick\'y \cite{10problems}]\label{tripartite}
	If $G$ is a $K_4$-free tripartite graph on $n$ vertices, then $\texttt{bb}(G) \leq \frac{n^2}{9}$.
\end{theorem}

\begin{proof}
	Let $A$, $B$, and $C$ be the $3$-partition of $G$ with respective relative sizes (normalized by $n$) denoted as $a$, $b$, and $c$.
	If any of $a$, $b$ or $c$ is larger than ${1}/{2}$, then the wanted sparse balanced partition exists, because by Tur\'an's theorem the set of ${n}/{2}$ vertices outside an independent set of size ${n}/{2}$ contains at most ${n^2}/{12}$ edges.

	In the other cases, we may assume that the graph is complete tripartite as adding the remaining edges does not simplify the problem.
	Consider the following random process of obtaining a balanced partition:
	\begin{enumerate}
		\item Randomly choose one of $A$, $B$ or $C$.
		\item Increase the other sets to size $\frac{n}{2}$ by randomly adding vertices from the chosen set.
	\end{enumerate}
	Assume that we chose $A$ in the first step.
	Then, the expected number of class-edges (normalized by $n^2$) is exactly
	$$\left(\frac{\frac{1}{2} - b}{a}\right)\cdot ab + \left(\frac{\frac{1}{2} - c}{a}\right)\cdot ac = \left(\frac{1}{2}-b\right)b + \left(\frac{1}{2}-c\right)c.$$

	If we now incorporate the random choice from the first step into the formula, the expected number of class-edges changes to (keep in mind $a + b + c = 1$)
	$$\frac{1}{3} \left[\left(\frac{1}{2} - b\right) b + \left(\frac{1}{2} - c\right) c\right] +
		\frac{1}{3} \left[\left(\frac{1}{2} - a\right) a + \left(\frac{1}{2} - c\right) c\right]
		+ \frac{1}{3} \left[\left(\frac{1}{2} - a\right) a + \left(\frac{1}{2} - b\right) b\right] $$
$$		= \frac{1}{9} - \frac{2}{3}\left(a - \frac{1}{3}\right)^2 - \frac{2}{3}\left(b - \frac{1}{3}\right)^2 - \frac{2}{3}\left(c - \frac{1}{3}\right)^2 \leq \frac{1}{9}.$$
	Hence, there exists a balanced partition with at most $\frac{n^2}{9}$ class-edges, so ${\texttt{bb}(G) \leq \frac{n^2}{9}}$.
\end{proof}

\section{Graphs with a large pair of independent sets}\label{sec:2-ind}
In this section, we prove the following relaxed version of \autoref{main}:
\begin{theorem}\label{2-ind}
	If $G$ is a $K_4$-free graph on $n$ vertices with two disjoint independent sets ${I_1 \subseteq V(G)}$, $I_2 \subseteq V(G)$ satisfying $|I_1| + |I_2| \geq \frac{2n}{3}$, then $\texttt{bb}(G) \leq \frac{n^2}{9}$.
\end{theorem}

Notice that \autoref{2-ind} is strictly stronger than \autoref{tripartite}, as the union of the  two largest parts of tripartite partition is always greater than ${2n}/{3}$.
\autoref{2-ind} is non-trivial, as a random extension of $I_1$ and $I_2$ is insufficient even when $|I_1| = |I_2|$.

By \autoref{parity}, it is enough to consider only graphs of order divisible by 6.
If $G$ has a pair of independent sets of total size \emph{at least} ${2n}/{3}$, it also has a pair of independent sets of total size \emph{exactly} ${2n}/{3}$.

\begin{definition}
	A graph~$G$ on $n$ vertices is \emph{convenient} if it is $K_4$-free and contains two disjoint independent subsets $I_1$, $I_2$ satisfying $|I_1| + |I_2| = \frac{2n}{3}$. W.l.o.g., we always assume that $|I_2| \geq |I_1|$, which implies that $|I_2| \geq \frac{n}{3}$.  We    write  $c_1 := \frac{n}{2} - |I_2|,\  c_2 := \frac{n}{6}$ and  $R:= V(G) \setminus (I_1 \cup I_2)$.
\end{definition}

\begin{observation}\label{Neighborhoods}
If $G$ is a convenient graph on $n$ vertices, then there exists a balanced bipartition of $G$ with at most $\frac{n^2}{9} - \left(|I_2|-\frac{n}{3}\right)^2$ class-edges, or for every edge~$(u, v) \in E(R)$ at least one of the following inequalities holds:
	\begin{itemize}
		\item $|N(u, v) \cap I_1| \leq c_1$,
		\item $|N(u, v) \cap I_2| \leq c_2.$
	\end{itemize}
\end{observation}
\begin{proof}
	Assume that there is an edge $(u,v) \in E(R)$, for which both of these inequalities  are false.
	Let $B$ be a subset of $N(u, v) \cap I_2$ of size exactly $c_2$, and let $A$ be a subset of $N(u, v) \cap I_1$ of size exactly $c_1$.
	Since $G$ is $K_4$-free, we have  $e(N(u, v)) = 0$, and, consequently, $e(A, B) = 0$.
	The set $I_2 \cup A$ is of size exactly $\frac{n}{2}$, and it  has at most
	$$e(I_2 \setminus B, A) \leq (|I_2| - |B|)|A| = (|I_2| - c_2)c_1 =
	\left(|I_2| - \frac{n}{6}\right)\left(\frac{n}{2}- |I_2| \right)
	= \frac{n^2}{36} - \left(|I_2| - \frac{n}{3}\right)^2$$
	edges.
	Additionally, by Tur\'an's theorem we have $e(V(G)\setminus(I_2 \cup A)) \leq {n^2}/{12}$. 
	Therefore, the total number of class-edges in the obtained balanced bipartition is at most $\frac{n^2}{9} - \left(|I_2| - \frac{n}{3}\right)^2$.
\end{proof}

\begin{lemma}\label{triangle-free}
	If $G$ is a convenient graph on $n$ vertices, such that $G[R]$ is triangle-free, then there exists a balanced bipartition of $G$ with at most $\frac{n^2}{9} - \frac{3}{4}\left(|I_2|-\frac{n}{3}\right)^2$ class-edges.
\end{lemma}

\begin{proof}
	Run the following greedy algorithm starting with $A = I_1$, $B = I_2$ and $R' = R$.
	\begin{enumerate}
		\item If there are no edges in $R'$, randomly assign vertices remaining in $R'$ to $A$ and $B$, such  that  ${|A| = |B| = \frac{n}{2}}$, and STOP.
		\item If $|A| = \frac{n}{2}$, then add $R'$ to $B$ and STOP.
		\item If $|B| = \frac{n}{2}$, then add $R'$ to $A$ and STOP.
		\item Choose an edge~$(u, v)$ from $R'$.
		\item If $|N(u, v) \cap I_1| \leq c_1$, then assign $A = A \cup \{u, v\}$, $R' = R' \setminus \{u, v\}$, and go to Step 1.
		\item If $|N(u, v) \cap I_2| \leq c_2$, then assign $B = B \cup \{u, v\}$, $R' = R' \setminus \{u, v\}$, and go to Step 1.
	\end{enumerate}
	By \autoref{Neighborhoods}, there exists a balanced bipartition of $G$ with at most $\frac{n^2}{9} - \left(|I_2|-\frac{n}{3}\right)^2$ class-edges, or one of the conditions in Steps 5 and 6  always holds. Since there is a finite number of edges in $R$, the algorithm will eventually stop.

	We now perform case analysis, based on which step the algorithm has exited in.
\setcounter{case}{0}
	\begin{case}
		Algorithm has exited in Step 2.
	\end{case}
	Assume that  $(u, v)$ got assigned to $A$ in Step 5.
	By the precondition of Step 5, we have
	\begin{equation}\label{step 5}
		e(I_1, \{u, v\}) \leq |I_1| + c_1.
	\end{equation}
	We  obtain an identical bound for every pair of vertices that was added to $A$ in Step~5.
	Since the algorithm has exited in Step 2, all vertices in $A$ that were not originally in $I_1$ must have been added in Step 5.
	For each of the $\frac{1}{2} (\frac{n}{2} - |I_1|)$ pairs added in Step 5, \eqref{step 5}~holds.
	Summing all of these bounds together, we obtain
	$$e(I_1, A \setminus I_1) \leq \frac{1}{2} \left(\frac{n}{2} - |I_1|\right)(|I_1| + c_1).$$

	Using this inequality, the triangle-freeness of $G[A \setminus I_1]$ and $G[B \setminus I_2]$ (which follows from the triangle-freeness of $G[R]$),
	and Tur\'{a}n's theorem,
	we  provide an upper bound on the total number of class-edges in the partition $(A, B)$
\begin{align*}
	e(A) + e(B) &= e(I_1, A \setminus I_1) + e(A \setminus I_1) + e(I_2, B \setminus I_2) + e(B \setminus I_2)\\
	 &\leq \frac{1}{2}\left(\frac{n}{2} - |I_1|\right)(|I_1| + c_1) + \frac{1}{4}\left(\frac{n}{2} - |I_1|\right)^2 + |I_2|\left(\frac{n}{2} - |I_2|\right) + \frac{1}{4}\left(\frac{n}{2} - |I_2|\right)^2\\
	 &= \left(|I_2| - \frac{n}{6}\right)\left(\frac{7n}{12} - |I_2|\right) + \frac{1}{4}\left(|I_2| - \frac{n}{6}\right)^2 + |I_2|\left(\frac{n}{2} - |I_2|\right) + \frac{1}{4}\left(\frac{n}{2} - |I_2|\right)^2\\
	 &= \frac{-3|I_2|^2}{2} + \frac{11n|I_2|}{12} - \frac{n^2}{36} \leq \frac{n^2}{9} - \frac{3}{4}\left(|I_2|-\frac{n}{3}\right)^2,
\end{align*}
	where the last inequality is equivalent to
	$$\left(\frac{3|I_2|}{4} - \frac{n}{6}\right)\left(|I_2| - \frac{n}{3}\right) \geq 0,$$
	which  holds since $|I_2| \geq \frac{n}{3}$.

	\begin{case}
		Algorithm has exited in Step 3.
	\end{case}
	In this case, we can analogously obtain the following bound:
\begin{align*}
	&e(A) + e(B) = e(I_1, A \setminus I_1) + e(A \setminus I_1) + e(I_2, B \setminus I_2) + e(B \setminus I_2)\\
	&\leq |I_1|\left(\frac{n}{2} - |I_1|\right) + \frac{1}{4}\left(\frac{n}{2} - |I_1|\right)^2 + \frac{1}{2}\left(\frac{n}{2} - |I_2|\right)(|I_2| + c_2) + \frac{1}{4}\left(\frac{n}{2} - |I_2|\right)^2 \\
	&= \left(\frac{2n}{3} - |I_2|\right)\left(|I_2| - \frac{n}{6}\right) + \frac{1}{4}\left(|I_2| - \frac{n}{6}\right)^2 + \frac{1}{2}\left(\frac{n}{2} - |I_2|\right)\left(|I_2| + \frac{n}{6}\right) + \frac{1}{4}\left(\frac{n}{2} - |I_2|\right)^2\\
	&= -|I_2|^2 + \frac{2n|I_2|}{3} \leq \frac{n^2}{9} - \frac{3}{4}\left(|I_2|-\frac{n}{3}\right)^2,
\end{align*}
where the last inequality is equivalent to
	$$\frac{1}{4}\left(|I_2| - \frac{n}{3}\right)^2 \geq 0.$$
		\begin{case}
		Algorithm has exited in Step 1.
	\end{case}
	Let $B'$ be the set of vertices assigned to $B$ in Step 6 of the algorithm (by $A'$ denote the analogous set for Step 5).
	Since there exists a perfect matching in $B'$, and $G[R]$ is triangle-free, every vertex in $R \setminus B'$ is connected to at most ${|B'|}/{2}$ vertices in $B'$.
	This observation gives us the following inequality
	$$e(B \setminus (I_2 \cup B'), B') \leq \frac{|B'|}{2}\left(\frac{n}{2} - |I_2| - |B'|\right).$$
	As the algorithm exited in Step 1, there is no  edge in $G[B \setminus (I_2 \cup B')]$.
	Together with the  Tur\'{a}n's bound on $e(B')$, this gives us
	$$e(B \setminus I_2) \leq \frac{|B'|^2}{4} + \frac{|B'|}{2}\left(\frac{n}{2} - |I_2| - |B'|\right) = \frac{|B'|}{4}\left(n - 2|I_2| - |B'|\right).$$
	Similarly to the previous cases, we can also prove an upper bound on $e(I_2, B')$: 
	$$e(I_2, B') \leq \frac{|B'|}{2}(|I_2| + c_2).$$
	Finally, we have the following bound on the number of edges between $B \setminus (B' \cup I_2)$ and $I_2$:
	$$e(I_2, B \setminus (B' \cup I_2)) \leq |I_2| \left(\frac{n}{2} - |I_2| - |B'|\right).$$
	Summing all of these bounds together, we obtain
\begin{align*}
e(B) &= e(B \setminus I_2) + e(I_2, B') + e(I_2, B \setminus (B' \cup I_2))\\
	&\leq \frac{|B'|}{4}(n - 2|I_2| - |B'|) + \frac{|B'|}{2}\left(|I_2| + c_2\right) + |I_2|\left(\frac{n}{2} - |I_2| - |B'|\right)\\
	&= -\frac{|B'|^2}{4} - |B'|\left(|I_2| - \frac{n}{3}\right) + \frac{n|I_2|}{2} - |I_2|^2\\
	& =  \left(|I_2| - \frac{n}{3}\right)^2 + \frac{n|I_2|}{2} - |I_2|^2 - \left( \frac{|B'|}{2} -\left(|I_2| - \frac{n}{3}\right) \right)^2,
\end{align*}
	 which implies that
	\begin{equation}\label{max-b}
		e(B) \leq \left(|I_2| - \frac{n}{3}\right)^2 + \frac{n|I_2|}{2} - |I_2|^2.
	\end{equation}
	We  can   obtain an analogous bound on the number of edges in $A$:
	$$e(A) \leq \frac{|A'|}{4}(n - 2|I_1| - |A'|) + \frac{|A'|}{2}(|I_1| + c_1) + |I_1|\left(\frac{n}{2} - |I_1| - |A'|\right).$$
	This time, we substitute $|I_1| = \frac{2n}{3} - |I_2|$ and $c_1 = \frac{n}{2}-|I_2|$ to obtain
\begin{align*}
e(A) &\leq \frac{|A'|}{4}\left(-\frac{n}{3} + 2|I_2| - |A'|\right) + \frac{|A'|}{2}\left(\frac{7n}{6} - 2|I_2|\right) + \left(\frac{2n}{3} - |I_2|\right)\left(-\frac{n}{6} + |I_2| - |A'|\right)\\
	&= -\frac{|A'|^2}{4} + \frac{|A'|}{2}\left(|I_2|-\frac{n}{3} \right) - \frac{n^2}{9} + \frac{5n|I_2|}{6} - |I_2|^2\\
	& \frac{1}{4}\left(|I_2|-\frac{n}{3}\right)^2 - \frac{n^2}{9} + \frac{5n|I_2|}{6} - |I_2|^2 - \left( \frac{|A'|}{2} - \left(|I_2| - \frac{n}{3}\right) \right)^2,
\end{align*}
 implying 
	\begin{equation}\label{max-a}
		e(A) \leq \frac{1}{4}\left(|I_2|-\frac{n}{3}\right)^2 - \frac{n^2}{9} + \frac{5n|I_2|}{6} - |I_2|^2.
	\end{equation}

	Combining \eqref{max-b} and \eqref{max-a}, we bound the number of class-edges in the partition returned by the algorithm:
\begin{align*}
e(A) + e(B) &\leq \frac{5}{4}\left(|I_2| - \frac{n}{3}\right)^2 - \frac{n^2}{9} + \frac{4n|I_2|}{3} - 2|I_2|^2\\
	&= \frac{5}{4}\left(|I_2| - \frac{n}{3}\right)^2 - 2\left(|I_2| - \frac{n}{3}\right)^2 + \frac{n^2}{9} = \frac{n^2}{9} - \frac{3}{4}\left(|I_2| - \frac{n}{3}\right)^2 ,
\end{align*}
	which is exactly what we wanted to prove.
\end{proof}

\begin{definition}
	For a convenient graph $G$, let $T$ be the set of vertices of a maximal set of disjoint triangles in $R$ and $\texttt{cl}(G)$ be the graph defined as follows
	\begin{itemize}
		\item $V(\texttt{cl}(G)) = V(G)$,
		\item $E(\texttt{cl}(G)) = E(G) \setminus \{ (u, v): \{u, v\} \cap T \neq \emptyset\} \cup \{ (u, v): u \in T, v \in I_1 \cup I_2\}$.
	\end{itemize}
\end{definition}

\begin{lemma}\label{pay for triangles}
	For every convenient graph~$G$ and a bipartition $(A, B)$ of $\texttt{cl}(G)$ satisfying $I_1 \subseteq A$, $I_2 \subseteq B$, there exists a balanced bipartition $(A', B')$ of $G$, for which
	$$e_G(A') + e_G(B') \leq e_{\cl(G)}(A) + e_{\cl(G)}(B) + \frac{3}{4}\left(|I_2| - \frac{n}{3}\right)^2.$$
\end{lemma}
\begin{proof}
	Let $T_1 = A \cap T$ and $T_2 = B \cap T$.
	Define $(A', B')$ to be the balanced  partition obtained by randomly redistributing vertices from $T$ to $A \setminus T_1$ and $B \setminus T_2$.
	In $\cl(G)$ it holds
	$$e_{\cl(G)}(A') + e_{\cl(G)}(B') = e_{\cl(G)}(A) + e_{\cl(G)}(B),$$
	as vertices from $T$ are indistinguishable in $\cl(G)$.

	Notice that any edge that is present in $\cl(G)$ and is not present in $G$ (or vice versa) must have at least one endpoint in $T$.
	For that reason, in order to bound the difference between the number of class-edges in $G$ and in $\cl(G)$ in the partition $(A',B')$, it is enough to count the edges incident to $T$. 

	The $K_4$-freeness of $G$ implies
	$$e_G(B \setminus T, T) \leq \frac{2}{3}|B \setminus T| \cdot  |T|,$$
  as $T$ is a collection of vertex disjoint triangles, so every vertex outside of it must be disconnected from at least one vertex in every one of those triangles.
	Since we redistributed vertices in $T = T_1 \cup T_2$ randomly, we have
	$$e_G(B' \setminus T_2, T_2) \leq \frac{2}{3}|B' \setminus T| \cdot |T| \cdot \frac{|T_2|}{|T|} = \frac{2}{3}|B' \setminus T_2| \cdot |T_2|.$$
	Together with a Tur\'{a}n bound on $e_G(T_2)$, we get
\begin{align*}
e_{G}(B') - e_{\cl(G)}(B') &= e_G(T_2) + e_G(B' \setminus T_2, T_2) - |I_2||T_2| 
	\leq \frac{|T_2|^2}{3} + \frac{2}{3}|B' \setminus T_2||T_2| - |I_2||T_2| \\
	&= \frac{|T_2|^2}{3} + \frac{2}{3}\left(\frac{n}{2} - |T_2|\right)|T_2| - |I_2||T_2| 
	= - \frac{|T_2|^2}{3} - \left(|I_2|-\frac{n}{3}\right)|T_2| \leq 0.
\end{align*}
	For $A'$, we  similarly derive the following bound
	\begin{align*}
	& e_{G}(A') - e_{\cl(G)}(A') \leq - \frac{|T_1|^2}{3} - \left(|I_1| - \frac{n}{3}\right)|T_1| = -\frac{|T_1|^2}{3} + \left(|I_2| - \frac{n}{3}\right)|T_1|\\
	 = & \  \frac{3}{4}\left(|I_2| - \frac{n}{3}\right)^2 -\frac{1}{3}\left( |T_1|+\frac{n}{3}-|T_2|\right)^2 \leq  \  \frac{3}{4}\left(|I_2| - \frac{n}{3}\right)^2.
	\end{align*}
		Altogether, the two inequalities complete the proof of the lemma, as we get
	$$e_G(A') + e_G(B') \leq e_{\cl(G)}(A) + e_{\cl(G)}(B) + \frac{3}{4}\left(|I_2| - \frac{n}{3}\right)^2.$$
\end{proof}

\begin{proof}[Proof of \autoref{2-ind}]
For every convenient graph~$G$, by \autoref{triangle-free}, there exists a ba\-lanced bipartition $(A', B')$ of $\cl(G)$ with at most
$$\frac{n^2}{9} - \frac{3}{4}\left(|I_2| - \frac{n}{3}\right)^2$$ class-edges.
Partition $(A', B')$ can be then transformed by \autoref{pay for triangles} to a partition $(A, B)$ of $G$ with at most
$$\frac{n^2}{9}-\frac{3}{4}\left(|I_2| - \frac{n}{3}\right)^2+\frac{3}{4}\left(|I_2| - \frac{n}{3}\right)^2 = \frac{n^2}{9}$$
class-edges.
\end{proof}

\section{``Almost" complete, balanced tripartite graphs}\label{sec:almost}
We now present a proof for graphs that are ``almost" complete, balanced tripartite.

Our proof of the general case revolves around proving that potential counterexamples to \autoref{main} have to be ``almost" complete, balanced tripartite.
We use observations from this section to conclude that argument.

We begin with a well-known folklore lemma, and we provide a short proof for completeness.
\begin{lemma}\label{triangle-free packing}
	Let $a, b, c$ be fixed integers, such that $c \geq b \geq a$.
	If $G$ is a tripartite, triangle-free graph with a $3$-partition $A, B, C$, with respective sizes $a, b, c$, then $e(G)\le ac + bc$.
\end{lemma}

\begin{proof}
	Assume $e(A, B) > 0$.
  Fix an arbitrary edge~$(u, v) \in E(A, B)$.
	Let $m(u, v)$ be the number of pairs $s, t$, such that $s \in \{u, v\}$, $t \in C$ and $(s, t) \notin E(G)$.
	For every $w \in C$ it  holds that $(u, w) \notin E(G)$ or $(v, w) \notin E(G)$, which implies $m(u, v) \geq c$.

	This gives that $\sum_{(u, v) \in E(A,B)} m(u, v) \geq e(A, B) \cdot c$.
	Since the maximum degree in $G[A \cup B]$ is at most $b$, every  pair $(s, t)$ is counted  at most $b$ times.
	This means that
	$$ e(G) \leq ac + bc - \frac{e(A, B) \cdot c}{b} + e(A, B) \leq ac + bc.$$
\end{proof}

To formalize the intuition of what an ``almost" complete, balanced, tripartite graph is, we introduce the notion of $\varepsilon$-niceness. 
We denote by $T(G)$ the set of triangles in $G$. 

\begin{definition}\label{e-nice}
	For $\varepsilon >0$, a graph $G$ on $n$ vertices with $e$ edges and $t$ triangles is called \emph{$\varepsilon$-nice}, if it satisfies the following inequalities
	\begin{equation}\label{nice-e-old}
		\frac{2e}{n^2} \left(\frac{2}{3} - \frac{2e}{n^2}\right) < \varepsilon,
	\end{equation}

	\begin{equation}\label{nice-t-old}
		\frac{2e}{n^2}\left(\frac{2}{9} - \frac{6t}{n^3}\right) < \varepsilon,
	\end{equation}

	\begin{equation}\label{nice-avg}
		\frac{6\sum_{\{u, v, w\} \in T(G)} n - \deg(u, v) - \deg(u, w) - \deg(v, w)}{n^4} < \varepsilon,
	\end{equation}

	\begin{equation}\label{min-degree}
		\delta (G) \geq \frac{4n-1}{9}.
	\end{equation}
\end{definition}

\begin{lemma}\label{3-ind}
	If $G$ is an $\varepsilon$-nice counterexample to \autoref{main} with $\varepsilon < 10^{-2}$, then $G$ has disjoint independent sets $A_1$, $A_2$ and $A_3$ satisfying
	${|A_1| + |A_2| + |A_3| \geq n(1 - 5\varepsilon)}$
	and $n(\frac{1}{3} - 5 \varepsilon) < |A_i| < n(\frac{1}{3} + 5 \varepsilon)$ for every $i \in \{1, 2, 3\}$.
\end{lemma}

\begin{proof}
	Notice that we can assume $e \geq \frac{2n^2}{9}$, as otherwise a random partition would have at most $\frac{n^2}{9}$ class-edges in expectation.
	In addition, for $e \in \big[ \frac{2n^2}{9}, \frac{n^2}{4}\big]$, \eqref{nice-e-old} cannot hold.
	Because of that, we can write
	$$\frac{1}{2} \left(\frac{2}{3} - \frac{2e}{n^2}\right) < \varepsilon,$$
	which is equivalent to
	\begin{equation}\label{nice-e}
		n^2 \left(\frac{1}{3} - \varepsilon\right) < e.
	\end{equation}
	Now, combining \eqref{nice-e} with \eqref{nice-t-old} gives
	$$2\left(\frac{1}{3} - \varepsilon\right) \left(\frac{2}{9} - \frac{6t}{n^3}\right) < \varepsilon,$$
	which is equivalent to
	$$n^3 \left( \frac{1}{27} - \frac{\varepsilon}{12(\frac{1}{3} - \varepsilon)}\right) < t,$$
	and for $\varepsilon \leq \frac{1}{12}$ implies that
	\begin{equation}\label{nice-t}
		n^3 \left(\frac{1}{27} - \frac{1}{3} \varepsilon\right) < t.
	\end{equation}

	Using inequalities \eqref{nice-avg} and \eqref{nice-t}, we conclude that there exists a triangle $u, v, w$ in~$G$, for which it holds (provided $\varepsilon < 10^{-2}$)
	$$n - \deg(u, v) - \deg(u, w) - \deg(v, w) < \frac{n^4 \varepsilon}{6} \cdot \frac{1}{n^3(\frac{1}{27} - \frac{1}{3} \varepsilon)}
		< \frac{\varepsilon n}{\frac{2}{9} - 2 \varepsilon} < 5 \varepsilon n.
	$$

	By $K_4$-freeness, the sets $N(u, v)$, $N(u, w)$, and $N(v, w)$ are independent and pairwise disjoint.
	As a result, $G$ contains three disjoint independent sets $A_1, A_2, A_3$, such that $|A_1| + |A_2| + |A_3| > n(1 - 5 \varepsilon)$.
	Moreover, by \autoref{2-ind}, we can conclude
	$$n\left(\frac{1}{3} - 5\varepsilon\right) < |A_i| < n\left(\frac{1}{3} + 5\varepsilon\right)$$
	for every $i \in \{1,2,3\}$.
\end{proof}

\begin{lemma}\label{dense-pairs}
	Let $G$ be an $\varepsilon$-nice counterexample to \autoref{main}.
	Let $A_1$, $A_2$, $A_3$ be three vertex disjoint independent sets in $G$ satisfying $|A_1| + |A_2| + |A_3| \geq n(1 - 5\varepsilon)$ and 
	${n(\frac{1}{3} - 5 \varepsilon) < |A_i| < n(\frac{1}{3} + 5 \varepsilon)}$ for $i \in \{1, 2, 3\}$.
	Then, for every $B_1 \subseteq A_i$, $B_2 \subseteq A_j$, such that $i, j \in \{1, 2, 3\}$, $i \neq j$, ${|B_1| \geq |B_2| \geq \sqrt{3\varepsilon}n}$, it holds $e(B_1, B_2) > 0$.
\end{lemma}

\begin{proof}
	Denoting $S=V(G)\setminus(A_1\cup A_2 \cup A_3)$, by Tur\'an's theorem in the $K_4$-free graph $G[S]$ we have $e(S) \leq \frac{1}{3}|S|^2$. Therefore, if $e(B_1, B_2) = 0$, then 
\begin{align*}
e(G) &\leq |A_1| |A_2| + |A_1| |A_3| + |A_2| |A_3| - |B_1| |B_2| + |S| |A_1\cup A_2 \cup A_3| + \frac{1}{3}|S|^2  \\
&\leq \frac{1}{3}|A_1\cup A_2 \cup A_3|^2 - |B_1| |B_2| + |S| |A_1\cup A_2 \cup A_3| + \frac{1}{3}|S|^2  \\
&= \frac{1}{3}n^2 + \frac{1}{3}|S| |A_1\cup A_2 \cup A_3| - |B_1| |B_2| 
\leq \frac{1}{3}n^2 + \frac{5}{3} \varepsilon (1 - 5 \varepsilon)n^2 - 3 \varepsilon n^2 \leq n^2\left(\frac{1}{3}-\varepsilon\right).
\end{align*}
	This contradicts with \eqref{nice-e}.
\end{proof}

\begin{lemma}\label{assign-s}
	Let $G$ be an $\varepsilon$-nice counterexample to \autoref{main} with $\varepsilon < 10^{-4}$.
	Let $S = V(G) \setminus (A_1 \cup A_2 \cup A_3)$, where $A_1$, $A_2$, $A_3$ are vertex disjoint independent sets satisfying $|A_1| + |A_2| + |A_3| \geq n(1 - 5\varepsilon)$
	and $n(\frac{1}{3} - 5 \varepsilon) < |A_i| < n(\frac{1}{3} + 5 \varepsilon)$ for $i \in \{1, 2, 3\}$.
	Then for every $v \in S$ there exists $i \in \{1, 2, 3\}$ such that $e(v, A_i) < 96\varepsilon n$.
\end{lemma}

\begin{proof}
	Fix $v \in S$, and w.l.o.g.~assume that $i = 1$ is minimizing the value $e(v, A_i)$.
	Assume for a contradiction that $sn = e(v, A_1) \geq 96\varepsilon n$.
	Degree of any vertex in $G$ is at least $\frac{4n-1}{9}$.
  Making a trivial assumption $n \geq 4$, we can also say that every vertex is of degree at least $\frac{15n}{36} \leq \frac{4n - 1}{9}$.
  As a result, both $e(v, A_2)$ and $e(v, A_3)$ must be at least $n(\frac{1}{12} - s - 10 \varepsilon)$.

	Since the graph~$N(v) \cap (A_1 \cup A_2 \cup A_3)$ is tripartite and triangle-free, we use \autoref{triangle-free packing} to conclude that 
  \begin{equation}\label{with-s}
    e(A_1 \cup A_2 \cup A_3) \leq \frac{n^2}{3} - sn^2\left(\frac{1}{12} - s - 10 \varepsilon\right).
  \end{equation}

  We now perform case analysis based on possible values of $s$.
  \setcounter{case}{0}
  \begin{case}
    $s < \frac{1}{12} - s - 10 \varepsilon$.
  \end{case}
  In this case, it is clear that quadratic function on the right-hand side of \eqref{with-s} is maximized as $s$ is minimized.
	Since $s \geq 96\varepsilon$, 
\begin{align*}
	e(G) &\leq \frac{n^2}{3} - sn^2\left(\frac{1}{12} - s - 10 \varepsilon\right) + \frac{(5\varepsilon n)^2}{3} + 5\varepsilon(1 - 5\varepsilon)n^2
	\leq n^2\left(\frac{1}{3} - 3\varepsilon + {10200}\varepsilon^2\right),
\end{align*}	
	which for $\varepsilon < 10^{-4}$ implies that $e(G) \leq n^2\left(\frac{1}{3} - \varepsilon\right)$, and this contradicts \eqref{nice-e}.

  \begin{case}
    $s \geq \frac{1}{12} - s - 10 \varepsilon$.
  \end{case}
  In this case, we want to leverage the fact that $A_1$ was the set minimizing the value $e(v, A_i)$.
  Because of this, both of $e(v, A_2)$ and $e(v, A_3)$ must also exceed $sn$.
  We use this to improve \eqref{with-s} to

  \begin{equation*}
    e(A_1 \cup A_2 \cup A_3) \leq \frac{n^2}{3} - s^2n^2.
  \end{equation*}
  Using $s \geq \frac{1}{24} - 5\varepsilon$, we get
  \begin{align*}
    e(G) &\leq \frac{n^2}{3} - s^2n^2 + \frac{(5\varepsilon n)^2}{3} + 5\varepsilon(1 - 5\varepsilon)n^2
         \leq n^2\left(\frac{1}{3} - \frac{1}{24^2} + \frac{130}{24}\varepsilon - \frac{125}{3}\varepsilon^2\right).
  \end{align*}
  Similarly to Case 1, for $\varepsilon < 10^{-4}$ this implies $e(G) \leq n^2\left(\frac{1}{3} - \varepsilon\right)$, which contradicts \eqref{nice-e}.
\end{proof}

The previous lemmas, for sufficiently small $\varepsilon$, give us a deep insight into the structure of $\varepsilon$-nice counterexamples to \autoref{main}.
We know that $G$ contains three large, disjoint, independent sets, and every vertex that does not belong to them is sparsely connected to at least one of those sets.

The most natural approach is to assign each of those vertices to the set to which they are sparsely connected, and then treat the graph as if it were tripartite.
This idea is indeed correct, and we argue that the balanced bipartition obtained similarly as in \autoref{tripartite} has sufficiently few class-edges.

\begin{definition}
	Let $G$ be an $\varepsilon$-nice counterexample to \autoref{main} with ${\varepsilon < 10^{-4}}$.
	Let $A_1$, $A_2$, $A_3$ be disjoint independent sets in $G$ satisfying $|A_1| + |A_2| + |A_3| \geq n(1 - 5 \varepsilon)$
	and $n(\frac{1}{3} - 5 \varepsilon) < |A_i| < n(\frac{1}{3} + 5 \varepsilon)$ for $i \in \{1, 2, 3\}$, that exist by \autoref{3-ind}.
	For every $i \in \{1, 2, 3\}$, let $A_i'$ be the set containing $A_i$ and all vertices in $V(G) \setminus (A_1 \cup A_2 \cup A_3)$, which has at most $96 \varepsilon n$ edges to $A_i$. By \autoref{assign-s} and $\delta(G) \geq \frac{4n - 1}{9}$, every vertex is in exactly one set $A_i'$.
	We refer to the obtained partition $A_1'$, $A_2'$, $A_3'$ as \emph{spotty partition}.
\end{definition}

The next lemma formalizes the intuition that adding edges inside the sets of a spotty partition is unlikely to make the problem more difficult,
as we need to remove significantly more edges between the sets in order to ensure $K_4$-freeness.

\begin{lemma}\label{missing edges}
	Let $G$ be an $\varepsilon$-nice graph with $\varepsilon < 10^{-4}$ and $A_1'$, $A_2'$, $A_3'$ its spotty partition.
	Let $(u, v)$ be an edge contained in $A_i'$ for some $i \in \{1,2,3\}$. Then for the other parts $A_j'$ and $A_k'$ of the spotty partition it holds 
	$${e(\{u, v\}, A_j' \cup A_k') \leq 2(|A_j'| + |A_k'|) - n\left(\frac{1}{3} - 5 \varepsilon - \sqrt{3\varepsilon}\right)}.$$
\end{lemma}

\begin{proof}
	Notice that $e(N(u, v) \cap A_j', N(u, v) \cap A_k') = 0$, as otherwise there would be a $K_4$ in~$G$.
	Together with \autoref{dense-pairs}, we can infer that one of $N(u, v) \cap A_j'$ and $N(u, v) \cap A_k'$ is smaller than $\sqrt{3\varepsilon}n$
	(w.l.o.g.~assume it is the first one).
	Then,
	$$e(\{u, v\}, A_j') \leq 2|A_j'| - (|A_j'| - \sqrt{3 \varepsilon}n) \leq 2|A_j'| - n \left(\frac{1}{3} - 5 \varepsilon - \sqrt{3\varepsilon}\right),$$
	which also means
	$$e(\{u, v\}, A_j' \cup A_k') \leq 2(|A_j'| + |A_k'|) - n\left(\frac{1}{3} - 5 \varepsilon - \sqrt{3 \varepsilon}\right),$$
	as required.
\end{proof}

\begin{theorem}\label{nice}
	\autoref{main} holds for every $\varepsilon$-nice graph $G$ with $\varepsilon < 10^{-4}$.
\end{theorem}

\begin{proof}
	Let $A_1', A_2', A_3'$ be a  spotty partition of $G$.
	Let $M$ be a  maximal matching, whose edges are  inside $A_1'$, $  A_2'$ and $ A_3'$ (not necessarily all in the same class). Note that by the maximality of $M$, every edge inside of one of the classes $A_1'$, $ A_2'$  and $ A_3'$ either belongs to $M$ or is incident with an edge in $M$.
	Denote 
	$$m(A_i', A_j') := |A_i'||A_j'| - e(A_i', A_j').$$

	Define a partition,  which we obtain by randomly distributing vertices from $A_1'$ to $A_2'$ and to $A_3'$, until they reach size $\frac{n}{2}$.
	Notice that this approach is sound, as each set of the spotty partition is of size at most $n (\frac{1}{3} + 10 \varepsilon)$, which is  less than $\frac{n}{2}$.
	We obtain a balanced partition, in which the expected number of class-edges is at most
	
\begin{align*}
&e(A_1') + e(A_2') + e(A_3') + \left(\frac{\frac{n}{2} - |A_2'|}{|A_1'|}\right) e(A_1', A_2') + \left(\frac{\frac{n}{2} - |A_3'|}{|A_1'|}\right) e(A_1', A_3')\\
	&= e(A_1') + e(A_2') + e(A_3') + \left(\frac{\frac{n}{2} - |A_2'|}{|A_1'|}\right)(|A_1'||A_2'| - m(A_1', A_2')) \\
	 &\qquad + \left(\frac{\frac{n}{2} - |A_3'|}{|A_1'|}\right)(|A_1'||A_3'| - m(A_1', A_3'))\\
	&\leq e(A_1') + e(A_2') + e(A_3') + \left(\frac{n}{2} - |A_2'|\right)|A_2'| + \left(\frac{n}{2} - |A_3'|\right)|A_3'| \\
	 &\qquad - \left(\frac{\frac{1}{6} - 10 \varepsilon}{\frac{1}{3} + 10 \varepsilon}\right) (m(A_1', A_2') + m(A_1', A_3')) \\
	 &\leq e(A_1') + e(A_2') + e(A_3') + \left(\frac{n}{2} - |A_2'|\right)|A_2'| + \left(\frac{n}{2} - |A_3'|\right)|A_3'| \\
	 &\qquad - \left(\frac{1}{2}-45\varepsilon\right) (m(A_1', A_2') + m(A_1', A_3')).
\end{align*}

	Now, the expected number of class-edges when the initial set is picked randomly (instead of $A_1'$) is at most
\begin{align*}
	&e(A_1') + e(A_2') + e(A_3') + \frac{2}{3}\left(\frac{n}{2} - |A_1'|\right)|A_1'| + \frac{2}{3}\left(\frac{n}{2} - |A_2'|\right)|A_2'| + \frac{2}{3}\left(\frac{n}{2} - |A_3'|\right)|A_3'| \\
	&\qquad - \frac{2}{3}\left(\frac{1}{2}-45\varepsilon\right)(m(A_1', A_2') + m(A_1', A_3') + m(A_2', A_3')) \\
	&\leq \frac{n^2}{9} + e(A_1') + e(A_2') + e(A_3') - \left(\frac{1}{3} - 30\varepsilon\right) (m(A_1', A_2') + m(A_1', A_3') + m(A_2', A_3')).
\end{align*}
	We can now see that this partition is sparse enough, as long as
	\begin{equation}\label{too much missing}
		e(A_1') + e(A_2') + e(A_3') \leq \left(\frac{1}{3} - 30\varepsilon\right) (m(A_1', A_2') + m(A_1', A_3') + m(A_2', A_3')).
	\end{equation}

	Notice that by \autoref{assign-s} the maximum degree in $G[A_i']$ is bounded by $96 \varepsilon n + 5 \varepsilon n$, which means
	$$e(A_1') + e(A_2') + e(A_3') \leq 202 \varepsilon |M| n.$$
	We now apply \autoref{missing edges} to every edge in $M$.
	Since $M$ is a matching, every missing edge between a pair of $A_1', A_2'$ and $A_3'$ is counted at most twice.
	This gives a lower bound on the total number of missing edges
	$$\frac{1}{2} |M| n\left(\frac{1}{3} - 5 \varepsilon - \sqrt{\varepsilon}\right) \leq m(A_1', A_2') + m(A_1', A_3') + m(A_2', A_3').$$
	We can now see that \eqref{too much missing} is true, whenever
	$$202 \varepsilon \leq \frac{1}{2}\left(\frac{1}{3} - 5 \varepsilon - \sqrt{\varepsilon}\right)\left(\frac{1}{3} - 30\varepsilon\right),$$
	which holds for $\varepsilon < 10^{-4}$.
\end{proof}

\section{Flag algebras}\label{sec:flags}
To complete the proof for the general case, we employ the powerful framework of flag algebras developed by Razborov in \cite{flag-algebras}.

Flag algebras are used to obtain bounds on asymptotic densities of induced substructures in combinatorial objects.
As of now, the most frequent applications were in graphs (\cite{Grz1, graph4, graph2, graph3}), but flags also proved to be useful in other use cases, such as oriented graphs (\cite{orgraph2, raz2}), hypergraphs (\cite{hyper1, hyper2}), hypercubes (\cite{cube1}) or permutations (\cite{permutations}). The method was also successfully applied to problems that are not expressed in terms of subgraph densities, but can be translated to such a setting by introducing colors on vertices or edges in a particular way (\cite{coloredflags, ramsey1}).

The process of obtaining bounds using the flag algebra method can be partially automated with the help of semidefinite programming.
A proper introduction to this framework is beyond the scope of this paper.
For this purpose, we refer the interested readers to the original paper by Razborov \cite{flag-algebras} or surveys \cite{flagsurvey3, flagsurvey1, flagsurvey2}.

For the rest of this section we focus on defining the notation that we use for flags, which is very similar to the one used in \cite{10problems}.

We only operate on the flags from the theory of simple, undirected, $2$-vertex-colored $K_4$-free graphs.
We color the vertices with blue and red.

We denote the flags from the algebra of the unlabeled flags by directly depicting the graph using round vertices.
For instance, this is the flag for the unlabeled triangle with two blue vertices and a single red vertex:
$$\Triangle{blue}{blue}{red}.$$

To denote a partially labeled flag $(G, \theta)$ from algebra $\mathcal{A}^\theta$, we depict $G$ and signify the injective mapping $\theta$ by writing the labels inside the vertices.
The vertices in $\operatorname{im}(\theta)$ are also square, rather than round.
For example, this is a labeled polychromatic edge with a single unlabeled red vertex connected to the blue endpoint of the edge
$$\EdgeL{blue}{red}{red}.$$

We also use the white color as a wildcard for any color.
More formally, a white vertex is syntactic sugar for the sum of the flags which are isomorphic to the template up to the color of this vertex.
When only one vertex in the template is white, this is equivalent to a sum of two flags, where we replace the white vertex with a blue one and a red one.
For example
$$\Triangle{blue}{blue}{white} = \Triangle{blue}{blue}{red} + \Triangle{blue}{blue}{blue}.$$
When the number of vertices increases, the sum gets longer. For instance
$$\Triangle{white}{white}{white} = \Triangle{red}{red}{red} + \Triangle{blue}{red}{red} + \Triangle{blue}{blue}{red} + \Triangle{blue}{blue}{blue}.$$

\section{General case}\label{sec:general}
In this section we leverage prior observations and the flag algebraic framework to conclude the proof in the general case.

\begin{definition}
  For a $K_4$-free graph $G$, we define its \emph{degree coloring} by coloring each vertex $v \in V(G)$ red if $\deg(v) \geq {|V(G)|}/{2}$, and blue if $\deg(v) < {|V(G)|}/{2}$.
\end{definition}

\begin{theorem}\label{best-basic}
  If $G$ is a $K_4$-free graph on $n$ vertices, then ${\texttt{bb}(G) \leq \frac{n^2}{9}}$.
\end{theorem}

\begin{proof}
  Let $G$ be the smallest (with respect to the number of vertices) counterexample with the degree coloring.
  By \autoref{min-degree-obs} we can assume that $G$ has no vertices of degree less than $({4n - 1})/{9}$.

  Let $W$ be the $2$-vertex-colored graphon obtained by considering the infinite sequence of increasing blowups of $G$.
  For the rest of the proof, we define flag inequalities that $W$ would have to satisfy and use semidefinite programming to prove that $W$ cannot exist.

  To obtain the first inequality, assume that there is at least one blue vertex in $G$, say $v$. By the definition of the coloring, the degree of $v$ is less than ${n}/{2}$.
  Consider a balanced bipartition, where one part is a random extension of $N(v)$ to size ${n}/{2}$ and the second one is its complement.
  Since $G$ is a counterexample to \autoref{main}, we know that there are at least ${n^2}/{9}$ class-edges inside this bipartition.
  This means that if we choose $v$ at random, then the expected number of edges also cannot be below ${n^2}/{9}$.
  We can write this as the following inequality for $G$
  \begin{equation}\label{cut-on-vertex}
    \frac{n^2}{9} \leq e(N(v)) +  e(N(v), V(G) \setminus N(v)) \cdot \frac{\frac{n}{2} - |N(v)|}{n - |N(v)|} 
    + e(V(G) \setminus N(v)) \cdot \left(\frac{\frac{n}{2}}{n - |N(v)|}\right)^2.
  \end{equation}

  Notice that this inequality heavily relies on the assumption that $v$ is blue. 
  It is clear that such inequality must hold for any $k$-blowup of $G$, which we can use to define an analogous relationship for $W$.
  To do this, we rely on the following correspondence between limits of densities in $k$-blowups of $G$ and flags of type $\RootedVertex{blue}$:

  \begin{itemize}
    \item $\frac{|N(v)|}{n} \to \VertexM{blue}{white}$,
    \item $\frac{2e(N(v))}{n^2} \to \VertexMM{blue}{white}{white}$,
    \item $\frac{|V(G) \setminus N(v)|}{n} \to \VertexN{blue}{white}$,
    \item $\frac{2e(V(G) \setminus N(v))}{n^2} \to \VertexNN{blue}{white}{white}$,
    \item $\frac{2e(N(v), V(G) \setminus N(v))}{n^2} \to \VertexNM{blue}{white}{white}$,
  \end{itemize}
  which gives that $W$ must satisfy
  \begin{equation*}
    \begin{aligned}
      \frac{2}{9} \leq & \VertexMM{blue}{white}{white} + \frac{\frac{1}{2} - \VertexM{blue}{white}}{\VertexN{blue}{white}} 
      \cdot \VertexNM{blue}{white}{white} + \VertexNN{blue}{white}{white} \cdot \left(\frac{\frac{1}{2}}{\VertexN{blue}{white}}\right)^2.
    \end{aligned}
  \end{equation*}
  This can be rewritten as
  \begin{equation*}
    \begin{aligned}
      0 \leq & \left(\VertexMM{blue}{white}{white} - \frac{2}{9}\right) \cdot \VertexN{blue}{white}^2 + \left(\frac{1}{2} - \VertexM{blue}{white}\right) \cdot \VertexN{blue}{white}
      \cdot \VertexNM{blue}{white}{white} + \frac{1}{4}\VertexNN{blue}{white}{white},
    \end{aligned}
  \end{equation*}
  and then averaged
  \begin{equation}\label{vertex-cut}
    \begin{aligned}
      0 \leq \unlab{
        \left(\VertexMM{blue}{white}{white} - \frac{2}{9}\right) \cdot \VertexN{blue}{white}^2 + \left(\frac{1}{2} - \VertexM{blue}{white}\right) \cdot \VertexN{blue}{white}
        \cdot \VertexNM{blue}{white}{white} + \frac{1}{4}\VertexNN{blue}{white}{white}
      }.
    \end{aligned}
  \end{equation}

  It is easy to see that if \eqref{vertex-cut} did not hold for $W$, then there would exist some $k$ (possibly very large), such that \eqref{cut-on-vertex} would not hold for the $k$-blowup of $G$.
  Notice that \eqref{vertex-cut} is still well-defined and holds even when there are no blue vertices in $G$ (e.g., in a complete balanced tripartite graph).

  We now define similar partitions, but based on more non-trivial roots (edge and cherry).
  Assume that there exists at least one edge in $G$ whose both endpoints are red, and fix one of them $(u, v)$.
  We now partition $V(G)$ into following parts based on $(u, v)$:
  \begin{equation*}
    L := N(u) \setminus N(v), \quad \quad  \quad \quad 
    R := N(u,v),  \quad \quad  \quad \quad 
    C := V(G) \setminus (L \cup R).
  \end{equation*}
 
  The next balanced bipartition is a random extension of $L$ and $R$ to size $\frac{n}{2}$ with vertices from $C$.
  Since  $L \subseteq V(G) \setminus N(v)$, the maximum size of $L$ is bounded by $\frac{n}{2}$, because $\deg(v) \geq {n}/{2}$.
  In addition, the size of $R$ is upper-bounded by ${n}/{2}$, because otherwise we would have an independent set of size ${n}/{2}$, which would give a balanced bipartition with at most ${n^2}/12$ class-edges by applying Tur\'an's theorem to its complement.
  Therefore, this bipartition is well-defined.
  A lower bound on the expected number of class-edges inside this partition gives the following inequality
  $$\frac{n^2}{9} \leq e(L) + e(R) + \frac{\frac{n}{2} - |L|}{|C|} e(L, C) + \frac{\frac{n}{2} - |R|}{|C|} e(R, C) +
  \left[\left(\frac{\frac{n}{2} - |L|}{|C|}\right)^2 + \left(\frac{\frac{n}{2} - |R|}{|C|}\right)^2\right]\cdot e(C).$$

  \begin{figure}
    \centering
    \includegraphics[width=0.6\textwidth]{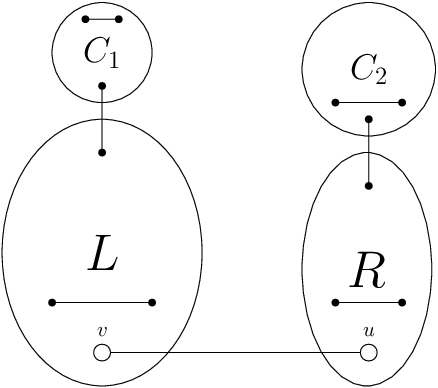}
    \caption{Different kinds of class-edges in the random bipartition based on the edge~$(u, v)$.}
  \end{figure}

  \noindent This time, limits of densities relevant to this partition are in the flag algebra of type $\RootedEdge{red}{red}$:
  \begin{itemize}
    \item $\frac{|L|}{n} \to \EdgeL{red}{red}{white}$,
    \item $\frac{2e(L)}{n^2} \to \EdgeLaL{red}{red}{white}{white}$,
    \item $\frac{|R|}{n} \to \EdgeB{red}{red}{white}$,
    \item $\frac{2e(R)}{n^2} = 0$,
    \item $\frac{|C|}{n}\to \EdgeR{red}{red}{white} + \EdgeN{red}{red}{white}$,
    \item $\frac{2e(L, C)}{n^2} \to \EdgeLaR{red}{red}{white}{white} + \EdgeLN{red}{red}{white}{white}$,
    \item $\frac{2e(R, C)}{n^2} \to \EdgeBR{red}{red}{white}{white} + \EdgeBN{red}{red}{white}{white}$,
    \item $\frac{2e(C)}{n^2} \to \EdgeNaN{red}{red}{white}{white} + \EdgeRN{red}{red}{white}{white} + \EdgeRaR{red}{red}{white}{white}$.
  \end{itemize}
  Before averaging, we get
  \begin{equation*}
    \begin{aligned}
      0 \leq &
      \left(\EdgeLaL{red}{red}{white}{white} - \frac{2}{9}\right)
      + \frac{\frac{1}{2} - \EdgeL{red}{red}{white}}{\EdgeR{red}{red}{white} + \EdgeN{red}{red}{white}} \left(\EdgeLaR{red}{red}{white}{white} + \EdgeLN{red}{red}{white}{white}\right) + \frac{\frac{1}{2} - \EdgeB{red}{red}{white}}{\EdgeR{red}{red}{white} + \EdgeN{red}{red}{white}} \left(\EdgeBR{red}{red}{white}{white} + \EdgeBN{red}{red}{white}{white}\right)                                                                    \\
             & + \left(\EdgeNaN{red}{red}{white}{white} + \EdgeRN{red}{red}{white}{white} + \EdgeRaR{red}{red}{white}{white}\right) 
             \cdot \left[ \left(\frac{\frac{1}{2} - \EdgeL{red}{red}{white}}{\EdgeR{red}{red}{white} + \EdgeN{red}{red}{white}}\right)^2
           + \left(\frac{\frac{1}{2} - \EdgeB{red}{red}{white}}{\EdgeR{red}{red}{white} + \EdgeN{red}{red}{white}}\right)^2\right],
    \end{aligned}
  \end{equation*}
  and after multiplying by the common denominator and unlabeling
  \begin{align}\label{first-edge-cut}
      0 \leq &
      \left\llbracket
      \left(\EdgeLaL{red}{red}{white}{white} - \frac{2}{9}\right)
      \left(\EdgeR{red}{red}{white} + \EdgeN{red}{red}{white}\right)^2 
      + \left(\frac{1}{2} - \EdgeL{red}{red}{white}\right)\left(\EdgeLaR{red}{red}{white}{white} + \EdgeLN{red}{red}{white}{white}\right)
      \left(\EdgeR{red}{red}{white} + \EdgeN{red}{red}{white}\right) 
      \right. \nonumber
      \\
             & + \left(\frac{1}{2} - \EdgeB{red}{red}{white}\right)
             \left(\EdgeBR{red}{red}{white}{white} + \EdgeBN{red}{red}{white}{white}\right)
             \left(\EdgeR{red}{red}{white} + \EdgeN{red}{red}{white}\right) \\
             & \left.
             + \left(\EdgeNaN{red}{red}{white}{white} + \EdgeRN{red}{red}{white}{white} + \EdgeRaR{red}{red}{white}{white}\right) 
             \cdot \left[ \left(\frac{1}{2} - \EdgeL{red}{red}{white}\right)^2
             + \left(\frac{1}{2} - \EdgeB{red}{red}{white}\right)^2 \right]
             \right\rrbracket. \nonumber
\end{align}
  For the next partition, we redefine sets $L$, $R$ and $C$ as:
  \begin{itemize}
    \item $L := N(u) \setminus N(v) \cup (V(G) \setminus (N(u) \cup N(v)))$,
    \item $R := N(v) \setminus N(u)$,
    \item $C := V(G) \setminus (L \cup R) = N(u,v)$.
  \end{itemize}
The obtained partition is well-defined, because $L \subseteq V(G) \setminus N(v)$, $R \subseteq V(G) \setminus N(u)$ and both vertices are red. After transformations analogous to the previous partition, this gives
  \begin{equation}\label{second-edge-cut}
    \begin{aligned}
      0 \leq & \left\llbracket 
      \left(\EdgeLaL{red}{red}{white}{white} + \EdgeLN{red}{red}{white}{white} + \EdgeNaN{red}{red}{white}{white} + \EdgeRaR{red}{red}{white}{white} - \frac{2}{9}\right)
      \EdgeB{red}{red}{white} \right. \\
             & + \left(\frac{1}{2} - \EdgeL{red}{red}{white} - \EdgeN{red}{red}{white}\right) 
             \left(\EdgeBL{red}{red}{white}{white} + \EdgeBN{red}{red}{white}{white}\right)
             + \left.
               \left(\frac{1}{2} - \EdgeR{red}{red}{white}\right)
               \EdgeBR{red}{red}{white}{white}
               \right\rrbracket.
    \end{aligned}
  \end{equation}
  Notice that there are no edges in $C$, which explains the absence of the final summand and lowers the power of the common denominator.

  The final partition we use is based on a cherry.
  We fix a red cherry $(u, v, w)$ (where $v$ is connected to both $u$ and $w$) and define sets $L$, $R$ and $C$ as
  \begin{itemize}
    \item $L := (N(u) \cup N(w)) \setminus N(v)$,
    \item $R := N(v) \setminus N(u) \cup (V(G) \setminus (N(u) \cup N(v) \cup N(w)))$,
    \item $C := V(G) \setminus (L \cup R) = N(u,v)$.
  \end{itemize}
  Since $L \subseteq V(G) \setminus N(v)$ and $R \subseteq V(G) \setminus N(u)$, it holds that $|L|, |R| \leq \frac{n}{2}$.
  After transformations analogous to the previous partition, i.e., randomly extending sets $A$ and $B$ with vertices from $C$ until the size of $\frac{n}{2}$, this partition gives
  \begin{align}\label{cherry-cut}
    &0 \leq \left\llbracket
  \left(\CherryRaR{red}{white} + \CherryLaR{red}{white} + \CherryLaL{red}{white} + \CherryLRaR{red}{white} + \CherryLaLR{red}{white} + \CherryLRaLR{red}{white} \right. \right. \nonumber \\
           & \qquad\quad + \left. \CherryNN{red}{white} + \CherryNM{red}{white} + \CherryMM{red}{white} + \CherryNaMR{red}{white} + \CherryMaMR{red}{white} - \frac{2}{9}\right)
           \left(\CherryLM{red}{white} + \CherryA{red}{white}\right) \\
           & + \left(\CherryLMaR{red}{white} + \CherryAaR{red}{white} + \CherryLaLM{red}{white} + \CherryLaA{red}{white} + \CherryLMaLR{red}{white} + \CherryLRaA{red}{white}\right)
           \left(\frac{1}{2} - \CherryR{red}{white} - \CherryL{red}{white} - \CherryLR{red}{white}\right) \nonumber \\
           & \left. + \left(\CherryLMaN{red}{white} + \CherryNaA{red}{white} + \CherryLMaMR{red}{white} + \CherryLMaM{red}{white} + \CherryMaA{red}{white}\right) 
           \left(\frac{1}{2} - \CherryM{red}{white} - \CherryMR{red}{white} - \CherryN{red}{white}\right)
             \right\rrbracket. \nonumber 
\end{align}
Similarly, there are no edges within $C$, which lowers the power of the common denominator and reduces the total number of vertices in the final linear combination of flags.
This is crucial for the tractability of the SDP instance that is used to finish this proof.

We use semidefinite programming solver to prove that a graphon satisfying \eqref{vertex-cut}, \eqref{first-edge-cut}, \eqref{second-edge-cut} and \eqref{cherry-cut} must also satisfy
\begin{equation}\label{must-be-nice}
  \Edge{white}{white}\left(\frac{2}{3} - \Edge{white}{white}\right) 
  + \Edge{white}{white}\left(\frac{2}{9} - \Triangle{white}{white}{white}\right) 
  + \unlab{1 - \TriangleLM{white}{white} - \TriangleMR{white}{white} - \TriangleLR{white}{white}} 
  < 10^{-9}.
\end{equation}
Exact calculations can be viewed at \url{https://github.com/ignacy123/k4-free-balanced-bipartitions}.

By \autoref{nice} we know that $G$ cannot be $\varepsilon$-nice with $\varepsilon = 10^{-5}$.
It is also easy to see that no $k$-blowup of $G$ can be $\varepsilon$-nice, for any $k \geq 1$.
However, \eqref{must-be-nice} implies that sufficiently large $k$-blowups of $G$ are $10^{-9}$-nice, which is a contradiction.
This shows that $W$ cannot exist and, in turn, neither can $G$.
It follows that the balanced bipartite distance of every $K_4$-free graph on $n$ vertices is at most $\frac{n^2}{9}$.
\end{proof}

\section{Acknowledgements}\label{sec:acknowledgement}
The computer-assisted part of this work is based on the flag program originally developed by Bernard Lidický (which was used in \cite{code-bernard}).
The authors are thankful for his permission for modifying and publishing the modified version of the program.

The authors would like to thank American Institute of Mathematics (AIM) for hosting the workshop ``Flag algebras and extremal combinatorics'', where three of the authors worked on this project.

The authors thank Bernard Lidick\'y and Felix Clemen for fruitful discussions.

\printbibliography
\end{document}